\RequirePackage{fix-cm}
\documentclass[11pt,envcountsect,smallextended]{svjour3}       
\smartqed
\usepackage{graphicx}
\usepackage{arevmath}
\usepackage[utf8]{inputenc}
\usepackage[T1]{fontenc}
\usepackage{lmodern}
\usepackage{amsmath}
\usepackage{amssymb}
\usepackage{marginnote}
\usepackage{xcolor}
\usepackage{fancyhdr}
\usepackage{esint}
\usepackage{tikz} 
\usepackage[arrow, matrix, curve]{xy}
\usepackage{hyperref}
\usetikzlibrary{fit,positioning} 
\usepackage{geometry}
\usepackage{pifont}
\usepackage{nccmath}
\usepackage{tocloft}
\usepackage{setspace}
\usepackage{graphicx}
\usepackage{calligra}

\usepackage[rose,novargreek,operator]{paper_diening}

\usepackage{blindtext}

\newcommand{\vertiii}[1]{{\vert\kern-0.25ex\vert\kern-0.25ex\vert #1 
		\vert\kern-0.25ex\vert\kern-0.25ex\vert}}
	
		\usepackage{color}
	\definecolor{Darkgreen}{rgb}{0,0.4,0}
	\usepackage{hyperref}
	\hypersetup{%
		pdfborder = {0 0 0},
		colorlinks,
		citecolor=blue,
		filecolor=black,
		linkcolor=blue,
		urlcolor=cyan!50!black!90
	}

\DeclareMathAlphabet\mathbfcal{OMS}{cmsy}{b}{n}
\DeclareMathAlphabet\mathbfscr{OMS}{mdugm}{b}{n}

\spnewtheorem{defn}[equation]{Definition}{\bfseries}{\upshape}
\spnewtheorem{prop}[equation]{Proposition}{\bfseries}{\upshape}
\spnewtheorem{thm}[equation]{Theorem}{\bfseries}{\upshape}
\spnewtheorem{cor}[equation]{Corollary}{\bfseries}{\upshape}
\spnewtheorem{rmk}[equation]{Remark}{\bfseries}{\upshape}
\spnewtheorem{lem}[equation]{Lemma}{\bfseries}{\upshape}
\spnewtheorem{expl}[equation]{Example}{\bfseries}{\upshape}

\numberwithin{equation}{section}
\renewcommand{\bfS}{\mathbfcal S}
\renewcommand{\bfB}{\mathbfcal B}
\renewcommand{\bfA}{\mathbfcal A}
\renewcommand{\bfX}{\mathbfcal X}

\begin{document}
	
\title{Note on the existence theory for pseudo-monotone evolution problems}

\author{A. Kaltenbach\and M. R\r{u}\v{z}i\v{c}ka}

\institute{A. Kaltenbach \at Institute of Applied Mathematics,
  Albert-Ludwigs-University Freiburg,
  Ernst-Zermelo-Straße 1, 79104 Freiburg,\\
  \email{alex.kaltenbach@mathematik.uni-freiburg.de}
  \\
  M. R\r{u}\v{z}i\v{c}ka \at Institute of Applied Mathematics,
  Albert-Ludwigs-University Freiburg,
  Ernst-Zermelo-Straße 1, 79104 Freiburg,\\
  \email{rose@mathematik.uni-freiburg.de} }
	
\date{Received: date / Accepted: date}

\maketitle

\begin{abstract}
  In this note we develop a framework which allows to prove an
  existence result for non-linear evolution problems
  involving time-dependent, pseudo-monotone operators. This abstract
  existence result is applicable to a large class of concrete problems
  where the standard theorem on evolutionary pseudo-monotone operators
  (cf.~Theorem~\ref{thm:b}) is not applicable. To this end we
  introduce the notion of Bochner pseudo-monotonicity, and Bochner
  coercivity, which are appropriate extensions of the concepts of
  pseudo-monotonicity and coercivity to the evolutionary
  setting. Moreover, we give sufficient conditions for these new
  notions, which are easily and widely applicable.
		
		
\end{abstract}

\section{Introduction}
\label{intro}
The theory of pseudo-monotone operators turned out to be a powerful
instrument in proving existence results for non-linear problems both
in the time-independent and the time-dependent setting.  The
celebrated main theorem on pseudo-monotone operators, stemming from
Brezis \cite{Bre68}, states the following\footnote{All notion are
  defined in Section~\ref{sec:2}.}
\begin{thm}\label{thm:a} 
  Given a reflexive, separable Banach space $V$, a right-hand side
  $f \in V^*$ and an operator $A:V\rightarrow V^*$ that it is
  pseudo-monotone, bounded and coercive, there exists a solution
  $u \in V$ of the problem
  \begin{align}\label{eq:s}
    \begin{alignedat}{2}
      Au&=f \quad &&\textrm { in } V^*,
    \end{alignedat}
  \end{align}
  i.e., the operator $A$ is surjective.
\end{thm}
Typical examples for such operators are sums of a monotone and a
compact operator (cf.~\cite{Zei90B}). Thus, the theory of pseudo-monotone operators
extends the theory of monotone operators due to Browder \cite{Bro63}
and Minty \cite{Min63}.

There exists also a time-dependent analogue of the above result
(cf.~\cite{Sho97}).
\begin{thm} \label{thm:b} Given an evolution triple
  $V{\hookrightarrow} H \cong H^*{\hookrightarrow}V^*$, a finite time
  horizon $I:=\left(0,T\right)$, an initial value
  $\boldsymbol{y}_0\in H$, a right-hand
  side\footnote{$p'=\frac{p}{p-1}$ for $p\in \left(1,\infty\right)$.}
  $\boldsymbol{f}\in L^{p'}(I,V^*)$, $1<p<\infty$, and a family of
  operators $A(t):V\rightarrow V^*$ such that the induced operator
  $\mathbfcal{A}:L^p(I,V)\rightarrow L^{p'}(I,V^*)$, given via
  $(\mathbfcal{A}\boldsymbol{x})(t):=A(t)(\boldsymbol{x}(t))$ for
  almost every $t\in I$ and all $\boldsymbol{x}\in L^p(I,V)$, is
  pseudo-monotone, bounded and coercive, there exists a solution
  $\boldsymbol{y}\in W^{1,p,p'}(I,V,V^*)$ of the initial value problem
  \begin{align}
    \begin{alignedat}{2}
      \frac{d\boldsymbol{y}}{dt}+\mathbfcal{A}\boldsymbol{y}
      &=\boldsymbol{f}\quad &&\text{ in }L^{p'}(I,V^*),
      \\
      \boldsymbol{y}(0)&=\boldsymbol{y}_0&&\text{ in }H.
      \label{eq:1.3}
    \end{alignedat}
  \end{align}
\end{thm}
This result is essentially a consequence of the main theorem on
pseudo-monotone perturbations of maximal monotone mappings, stemming
from Browder \cite{Bro68} and Brezis \cite{Bre68} (see also \cite[§32.4.]{Zei90B}). In doing so, one
interprets the time derivative
$\frac{d}{dt}:W^{1,p,p'}(I,V,V^*)\rightarrow L^{p'}(I,V^*)$ as a
maximal monotone mapping
\begin{align*}
  \begin{split}
    L\colon &D(L)\subseteq L^p(I,V)\rightarrow L^{p'}(I,V^*)\text{
      with }L\boldsymbol{x}:=\frac{d\boldsymbol{x}}{dt}\\\text{and
    }&D(L):=\{\boldsymbol{x}\in W^{1,p,p'}(I,V,V^*) \fdg
    \boldsymbol{x}(0)=\boldsymbol{0}\text{ in }H\}.
  \end{split}
\end{align*}
Note that the maximal monotonicity of the operator $L$ can basically
be traced back to the generalized integration by parts formula
(cf. Proposition \ref{2.11}). For details we refer to
\cite{Lio69,GGZ74,Zei90B,Sho97,Rou05,Ru04}.

To illustrate the applicability of the above existence results we
consider as prototypical applications the steady and unsteady motion
of incompressible shear dependent fluids in a bounded domain
$\Omega \subset \setR^3$. The motion in the steady and unsteady
situation, resp., is governed by
\begin{equation}
  \label{pfluid}
  \begin{aligned}
    -\divo \bS(\bD\bu)+[\nabla \bu]
    \bfu+\nabla\pi&=\ff\qquad&&\text{in }  \Omega,\\
    \divo\bu&=0\qquad&&\text{in } \Omega,
  \end{aligned}
\end{equation}
and
\begin{equation}
  \label{pfluid-un}
  \begin{aligned}
    \partial _t\bu -\divo \bS(\bD\bu)+[\nabla \bu]
    \bu+\nabla\pi&=\ff\qquad&&\text{in } I\times \Omega,\\
    \divo\bu&=0\qquad&&\text{in } I\times \Omega, \\
    \bu(0)&=\bu_0\qquad&&\text{in } \Omega,
  \end{aligned}
\end{equation}
respectively. Here $\bfu=(u_1,u_2,u_3)^\top$ is the fluid velocity,
$\bfD\bfu$ its symmetric gradient,
i.e., $\bfD\bfu =\frac 12 (\bfD\bfu+\bfD\bfu^\top)$,
$\bS(\bfD\bfu)=(\delta + \abs{\bfD\bfu})^{p-2}\bfD\bfu$,
$p \in (1,\infty)$, $\delta\ge 0$, the extra stress tensor, $\pi$ the
pressure, $\ff=(f_1,f_2,f_3)^\top$ the external body force, and
$\bfu_0$ the initial velocity. Here we used the notation
$([\nabla \bfu] \bfu)_i = \sum_{j=1}^3 u_j \partial _j u_i$,
${i=1,2,3}$, for the convective term. Setting
${V:=W^{1,p}_{0,\divo}(\Omega)}$, $H:=L^2_{\divo} (\Omega)$, the spaces
of solenoidal vector fileds from $W^{1,p}_0(\Omega)$ and
$L^2(\Omega)$, resp., the first two terms in \eqref{pfluid} define
operators $S,B\colon V \to V^*$ through
$\langle S \bfu, \bfv\rangle_V:= \int_\Omega \bS(\bfD\bfu) \cdot \bfD
\bfv\, dx$,
$\langle B \bfu,\bfv\rangle_V:= -\int_\Omega \bfu \otimes \bfu \cdot
\bfD \bfv\, dx$. Note that
$\int_\Omega \nabla \pi \cdot \bfv \, dx =0$ for $\bfv \in V $ and
sufficiently smooth $\pi$.  It is well known that the operator $S$ is
a strictly monotone, bounded, coercive and continuous 
(cf.~\cite{GGZ74,Lio69}), and that, due to the compact embedding
$W^{1,p}_0(\Omega)\vnor L^q(\Omega) $, $q<\frac {3p}{3-p}$, the
operator $B$ is compact, bounded and $\langle B \bfu,\bfu\rangle_V=0$
if $p>9/5$ (cf.~\cite{Lio69}). Thus, $A:=S+B$ fulfils the
assumptions of Theorem~\ref{thm:a}, yielding the existence of
solutions $\bfu \in V$ of \eqref{pfluid} if $p > 9/5$. To treat
the unsteady  problem \eqref{pfluid-un} note that the operators $S,B$ induce
operators
$\mathbfcal{S}, \mathbfcal{B}:L^p(I,V)\rightarrow L^{p'}(I,V^*)$
through $(\mathbfcal{S}\boldsymbol{x})(t):=S(\boldsymbol{x}(t))$,
$(\mathbfcal{B}\boldsymbol{x})(t):=B(\boldsymbol{x}(t))$ for almost
every $t\in I$ and all $\boldsymbol{x}\in L^p(I,V)$. The induced
operator $\bfS$ inherits the properties of the operator $S$
(cf.~\cite[Chapter 30]{Zei90B}), i.e., $\bfS$ is a strictly monotone,
bounded, coercive and continuous operator. On the other hand the
operator $\bfB$ is still bounded for sufficiently large $p$,
e.g.~$p>3$. However, the operator $\bfB$ is for no $p$
compact\footnote{This failure of the compactness is due to the fact
  that no information of the time derivative has been taken into
  account. If $\bfB$ is considered as an operator from
  $ W^{1,p,p'}(I,V,V^*)$ to $(W^{1,p,p'}(I,V,V^*))^*$ one can show
  that $\bfB$ is compact. However, on these spaces the operator $\bfA$
  is not coercive, and Theorem~\ref{thm:b} is again not applicable.},
since the embedding
$L^p(I,W^{1,p}_0(\Omega))\vnor L^p(I,L^q(\Omega)) $,
$q<\frac {3p}{3-p}$, is not compact. In fact, given $f_n\weakto f$ in
$L^p(I)$ containing no strongly convergent subsequence and
$\bg_n \weakto \bg$ in $W^{1,p}_0(\Omega)$, the sequence
$\bfu_n(t):=f_n(t)\bg_n$ is bounded in $L^p(I,W^{1,p}_0(\Omega))$ but
does not contain any strongly convergent subsequence in
$L^p(I,L^q(\Omega)) $.  Thus, Theorem \ref{thm:b} can not be applied
to the operator $\bfA:= \bfS +\bfB$. Nevertheless, it was already
observed in \cite[Theorems~2.5.1, 3.1.2]{Lio69} that, using and adapting the
ideas of the proof of Theorem~\ref{thm:b}, one can show the existence
of solutions $\bfu \in W^{1,p,p'}(I,V,V^*)$ of \eqref{pfluid-un} if
$p>11/5$.
	
In fact, this situation is prototypical and not an exception. To be
more precise, assume that the operator
$\mathbfcal{A}:L^p(I,V)\rightarrow L^{p'}(I,V^*)$ is induced by a
family of operators $A(t)\colon V\to V^*$, ${t\in I}$. We can not
expect $\mathbfcal{A}:L^p(I,V)\rightarrow L^{p'}(I,V^*)$ to be
pseudo-monotone even if the operators $\{A(t)\}_{t\in I}$ are
pseudo-monotone. This is due to the fact that
$\boldsymbol{x}_n\rightharpoonup \boldsymbol{x}$ in $L^p(I,V)$ in
general does not imply
$\boldsymbol{x}_n(t) \rightharpoonup \boldsymbol{x}(t)$ in $V$ for
a.e.~$t\in I$ (cf.~Remark \ref{rem:2.16}). Thus, the pseudo-monotonicity
of the operators $A(t)$ can not be used.  However, adapting the ideas
of \cite{Lio69,Hir1,Hir2} one can show the existence of solutions to
the evolutionary problem in many cases (cf.~\cite{BR17} for a
treatment using this approach). The drawback of this approach is that
additional technical assumptions on the spaces have to be made in
order to use the Aubin--Lions lemma. These additional assumptions are
not natural and in fact even not needed. It is the first purpose of
this paper to prove an abstract existence theorem for evolutionary
problems (cf.~Theorem~\ref{4.1}) that avoids unnecessary technical
assumptions and is applicable if Theorem~\ref{thm:b} is not applicable
but \cite[Theorem~2.5.1]{Lio69} could be adapted.   This leads us to the new notion of {\bf Bochner
  pseudo-monotone} operators. Moreover, based on the methods in
\cite{LM87,Shi97,Pap97,RT01,Rou05,BR17}, we are able to prove that the
induced operator $\mathbfcal{A}$ of a family of pseudo-monotone
operators $\{A(t)\}_{t\in I}$ satisfying appropriate coercivity and
growth conditions (cf.~conditions
(\hyperlink{C.1}{C.1})--(\hyperlink{C.5}{C.5}) which are easily
verifiable in concrete applications) is Bochner pseudo-monotone. Note
that the operator $\bfA=\bfS +\bfB$, introduced above for the
treatment of problem \eqref{pfluid-un}, satisfies these
conditions. Consequently, Theorem~\ref{4.1} or Corollary~\ref{4.2} are
applicable, in contrast to Theorem~\ref{thm:b}, which is not.
This observation applies to many other applications,
especially those where the inducing operator contains a compact
part. Thus, it seems that Bochner pseudo-monotonicity plays the same
role for nonlinear evolution problems as classical pseudo-monotonicity
for time-independent nonlinear problems. This is due to the fact that
Bochner pseudo-monotonicity takes into account the informations both
from the operator and the time derivative.  In the same spirit, we
introduce the notion of {Bochner coercivity}, which generalizes the
usual coercivity of the operator in the sense that it also takes into
account the information coming from the time derivative based on 
Gronwall's inequality. 

Consider now the problem \eqref{pfluid-un} without the convective
term. In this case Theorem~\ref{thm:b} is applicable. However, there
appears the restriction $p>6/5$, since the spaces $V, H$ have to form an evolution
triple. This lower bound is artificial and can be avoided if one works
on the intersection space $V\cap H$, at least in the case of a
monotone operator, as already observed in
\cite[Theorem~2.1.2bis]{Lio69}.  The second purpose of this paper is
to avoid artificial lower bounds based on problems with appropriate
embeddings. Thus, we introduce the notion of {\bf pre-evolution
  triples}, based on pull-back intersections, which generalize
evolution triples. Moreover, we develop the abstract theory of Bochner
pseudo-monotone operators immediately for pre-evolution triples.
	
The paper is organized as follows: In Section~\ref{sec:2} we introduce
the notation and some basic definitions and results concerning
Bochner-Lebesgue spaces, Bochner-Sobolev spaces, pull-back
intersections, evolution equations and introduce the notion of
pre-evolutions triples. In Section~\ref{sec:3} we introduce Bochner
pseudo-monotonicity and Bochner coercivity as appropriate extensions
of the concepts of pseudo-monotonicity and coercivity to the
evolutionary setting. In view of applications we will present some
sufficient conditions on operator families that imply these 
concepts. In Section~\ref{sec:4} we prove an existence result for
evolution equations with pre-evolution triples for abstract Bochner
pseudo-monotone and Bochner coercive operators as well as for
operators satisfying appropriate and easily verifiable sufficient
conditions. Section~\ref{sec:5} contains some illustrating example for
the before developed theory.

The paper is an extended and modified version of parts of the thesis
\cite{alex-master}. 
\section{Preliminaries}
\label{sec:2}
\subsection{Operators}
For a Banach space $X$ with norm $\|\cdot\|_X$ we denote by $X^*$ its
dual space equipped with the norm $\|\cdot\|_{X^*}$. The duality
pairing is denoted by $\langle\cdot,\cdot\rangle_X$. All occurring
Banach spaces are assumed to be real. By $D(A)$ we denote the domain
of definition of an operator $A:D(A)\subseteq X\rightarrow Y$, and
by $R(A):=\{Ax\mid x\in D(A)\}$ its range.
The following notions turn out to be useful in our investigation. 
\begin{defn}\label{2.1}
  Let $(X,\|\cdot\|_X)$ and $(Y,\|\cdot\|_Y)$ be Banach spaces. The
  operator $A:D(A)\subseteq X\rightarrow Y$ is said to be
  \begin{description}[\textrm{(iii)}]
  \item[\textrm{(i)}] \textbf{bounded}, if for all bounded
    $M\subseteq D(A)\subseteq X$ the image $A(M)\subseteq Y$ is
    bounded.
  \item[\textrm{(ii)}] \textbf{coercive}, if $Y=X^*$, $D(A)$ is
    unbounded and
    $\lim_{\substack{\|x\|_X \to \infty\\x \in D(A)}} \frac {\langle
      Ax,x\rangle _X}{\|x\|_X}=\infty$.
  \item[\textrm{(iii)}] \textbf{pseudo-monotone}, if $Y=X^*$, $D(A)=X$ and for
    each sequence $(x_n)_{n\in\mathbb{N}}\subseteq X$ with
    \begin{gather}
      x_n\overset{n\rightarrow\infty}{\rightharpoonup}x\text{ in }X
      , \label{eq:1.1}
      \\
      \limsup_{n\rightarrow\infty}{\langle Ax_n,x_n-x\rangle_X}\leq
      0 \label{eq:1.2}
    \end{gather}
    it follows that
    $\langle Ax,x-y\rangle_X\leq \liminf_{n\rightarrow\infty}{\langle
      Ax_n,x_n-y\rangle_X}$ for all $y\in X$.
  \end{description}
\end{defn}

Note that pseudo-monotonicity together with boundedness compensate for the
absence of weak continuity of the nonlinear operator $A$, as for a
sequence $(x_n)_{n\in\mathbb{N}}\subseteq X$ satisfying \eqref{eq:1.1} and \eqref{eq:1.2} it follows that
$Ax_n\overset{n\rightarrow\infty}{\rightharpoonup}Ax$ in $X^*$. Also
note that the conditions \eqref{eq:1.1} and \eqref{eq:1.2} are natural,
since if $(x_n)_{n\in\mathbb{N}}\subseteq X$ is a sequence of
appropriate Galerkin approximations of the problem \eqref{eq:s},
then \eqref{eq:1.1} is a consequence of the demanded coercivity and
\eqref{eq:1.2} can be derived directly from the properties of the
Galerkin approximation.

\subsection{Pre-evolution triple and pull-back intersections}
\label{sec:2.2}
Existence results for the initial value problem \eqref{eq:1.3}
(cf.~\cite{Bre68,Lio69,GGZ74,Zei90B,Sho97,Rou05,Ru04}) usually require
an evolution triple structure $(V,H,j)$, i.e., $(V,\|\cdot\|_V)$ is a
separable, reflexive Banach space, $(H,(\cdot,\cdot)_H)$ a Hilbert
space and $j:V\to H$ an embedding, such that $R(j)$ is dense in $H$,
e.g. $V=W^{1,p}_0(\Omega)$, $H=L^2(\Omega)$ and $j=\text{id}$ for
$p\ge\frac{2d}{d+2}$ fulfil these requirements. This evolution triple
structure is primarily needed for the validity of an integration by
parts formula (cf.~Proposition \ref{2.11}). Note that especially the
existence of a dense embedding in the definition of evolution triples
limits the scope of application, since for example $W^{1,p}_0(\Omega)$
does not embed into $L^2(\Omega)$ if $1<p<\frac{2d}{d+2}$.

Lions in \cite{Lio69} circumvented this limitation in the case of
autonomous monotone operators ${A:V\to V^*}$ by modifying
the 
coercivity condition $\langle Av,v\rangle_V\ge c\,\|v\|_V^p$ to
${\langle Av,v\rangle_V\ge c\,[v]_V^p}$, where $[\,\cdot\,]_V$ is a
semi-norm on $V$ such that $[v]_V+\lambda\|v\|_H\ge \beta \|v\|_V$ for
all $v\in V$ and for some $\lambda,\beta>0$. If $A$ is generated by
the $p$-Laplace operator for $p \in (1, \frac {2d}{d+2})$ the above
situation is realized by $V=W^{1,p}_0(\Omega)\cap L^2(\Omega)$,
equipped with
$\norm{\,\cdot\,}_V =\norm{\,\cdot\,}_{W^{1,p}_0(\Omega)} +
\norm{\,\cdot\,}_{L^2(\Omega)}$, $H=L^2(\Omega)$ and
$[\,\cdot\,]_V= \|\cdot\|_{W^{1,p}_0(\Omega)}$.

We proceed similarly and consider a separable, reflexive Banach space
$V$, a Hilbert space $H$ such that $V\cap H$ exists. We equip
$V\cap H$ with the canonical sum norm
$\norm{\,\cdot\,}_V =\norm{\,\cdot\,}_V + \norm{\,\cdot\,}_H$, such
that trivially $V\cap H$ embeds into $H$ densely, i.e., $V\cap H$ and
$H$ form an evolution triple. Moreover, we relax the coercivity
condition further, since we take into account the information coming
from the time derivative. To make this precise
(cf.~Definition~\ref{def:pre-evol}) we need some facts on
intersections of Banach spaces.  To this end, we want to propose an
alternative point of view, which turns out to be both quite
comfortable and exact, in the sense that we do not need to assume any
identifications of spaces and the amount of occurring embeddings is
marginal. Nonetheless, we emphasize that the standard definition of
intersections of Banach spaces (cf.~\cite{BS88}) is equivalent to our
approach and all the following assertions remain true if we use the
standard framework in \cite{BS88}.
 
\begin{defn}[Embedding]\label{6.1}
	Let $(X,\tau_X)$ and $(Y,\tau_Y)$ be topological vector spaces. The
	operator $j:X\rightarrow Y$ is said to be an \textbf{embedding} if
	it is linear, injective and continuous. In this case we use the
	notation
	\begin{align*}
	X\overset{j}{\hookrightarrow}Y.
	\end{align*}
	If $X\subseteq Y$ and $j=\text{id}_X:X\to Y$, then we write
	$X\hookrightarrow Y$ instead.
\end{defn}

\begin{defn}[Compatible couple]\label{6.2}
	Let $(X,\|\cdot\|_X)$ and $(Y,\|\cdot\|_Y)$ be Banach spaces such
	that embeddings $e_X:X\rightarrow Z$ and $e_Y:Y\rightarrow Z$ into a
	Hausdorff vector space $(Z,\tau_Z)$ exist. Then we call 
	$(X,Y):=(X,Y,Z,e_X,e_Y)$ a \textbf{compatible couple}.
\end{defn}

\begin{defn}[Pull-back intersection of Banach spaces]\label{6.3}
  Let $(X,Y)$ be a compatible couple. Then the operator
  $j:=e_Y^{-1}e_X:e_X^{-1}(R(e_X)\cap R(e_Y)) \rightarrow Y$
  is well-defined and we denote by
  \begin{align*}
    X\cap_j Y:=e_X^{-1}(R(e_X)\cap R(e_Y)) \subseteq X
  \end{align*}
  the \textbf{pull-back intersection of $X$ and $Y$ in $X$ with
    respect to $j$}. Furthermore, $j$ is said to be the
  \textbf{corresponding intersection embedding}. If $X,Y\subseteq Z$
  with $e_X=\text{id}_X$ and $e_Y=\text{id}_Y$, we set
  $X\cap Y:=X\cap_jY$.
\end{defn}

The next proposition shows that $j:X\cap_jY \to Y$ is indeed an
embedding if $X\cap_j Y$ is equipped with an appropriate norm. 
\begin{prop}
  \label{6.4}
  Let $(X,Y)$ be a compatible couple. Then $X\cap_j Y$ is a vector
  space and equipped with norm
  \begin{align*}
    \|\cdot\|_{X\cap_j Y}:=\|\cdot\|_X+\|j\cdot\|_Y
  \end{align*}
  a Banach-space. Moreover,  $j:X\cap_j Y\to Y$ is an embedding.
\end{prop}

\begin{proof}
  In \cite[Chapter 3, Theorem 1.3]{BS88} it is proved that
  $R(e_X)\cap R(e_Y)$ equipped with the norm
  $\|\cdot\|_{R(e_X)\cap R(e_Y)}:= \|e_X^{-1}(\cdot)\|_X +
  \|e_Y^{-1}(\cdot)\|_Y$ is a Banach space. Since
  $e_X^{-1}:R(e_X)\cap R(e_Y)\rightarrow X\cap_j Y$ is an isometry, if
  we equip $X\cap_j Y$ with the norm $\|\cdot\|_{X\cap_j Y}$, the space 
  $X\cap_j Y$ inherits the Banach space property of
  $R(e_X)\cap R(e_Y)$. The linearity and injectivity of  $j:X\cap_j
  Y\to Y$ are clear. The continuity of $j:X\cap_j Y \to Y$ follows directly from the
  definition of the norm in $X\cap_j Y$. \hfill$\square$
\end{proof}

\begin{rmk}[Fundamental properties of pull-back intersections]\label{6.5}
  \begin{description}[(iii)]
  \item[\textrm{(i)}] {\bf Consistency:} If $(X,\|\cdot\|_X)$ and
    $(Y,\|\cdot\|_Y)$ are Banach spaces such that an embedding
    ${j:X\rightarrow Y}$ exists, then $(X,Y)=(X,Y,Y,j,\text{id}_Y)$
    forms a compatible couple and it holds ${X\cap_j Y=X}$ with norm
    equivalence $\|\cdot\|_{X\cap_j Y}\sim\|\cdot\|_X$.
  \item[\textrm (ii)] {\bf Commutativity up to isomorphism:} For a
    compatible couple $(X,Y)$ the pull-back intersection
    $Y\cap_{j^{-1}} X$ of $X$ and $Y$ in $Y$ with respect to
    $j^{-1}=e_X^{-1}e_Y$ is well-defined as well. In addition,
    $j:X\cap_j Y\rightarrow Y\cap_{j^{-1}} X$ is an isometric
    isomorphism. Rephrased, pull-back intersections are thus
    commutative up to an isometric isomorphism.
  \item[\textrm (iii)] {\bf Associativity:} If
    $(X,Y)=(X,Y,Z,e_X,e_Y)$ and $(Y,W)=(Y,W,Z,e_Y,e_W)$ are compatible
    couples, $i:=e_W^{-1}e_X$, $j:=e_Y^{-1}e_X$ and $k:=e_W^{-1}e_Y$,
    then it holds $(X\cap_j Y)\cap_i W=X\cap_j(Y\cap_k W)$ and
    $\|\cdot\|_{(X\cap_j Y)\cap_i W}=\|\cdot\|_{X\cap_j(Y\cap_k W)}$.
  \end{description}
\end{rmk}

\begin{prop}\label{6.6}
  Let $(X,Y)$ be a compatible couple. Then it holds:
  \begin{description}[(iii)]
  \item[\text{(i)}] The graph
    $G(j):=\{(x,y)^\top \in X\times Y\fdg x\in X\cap_j Y, y=jx\}$ is closed in $X\times Y$.
  \item[\text{(ii)}] If $X$ and $Y$ are reflexive or separable, then
    $X\cap_j Y$ is as well.
  \item[\text{(iii)}] \textbf{First characterization of weak
      convergence in $X\cap_j Y$:} A sequence ${(x_n)_{n\in\mathbb{N}}\subset
    X\cap_jY}$ and $x \in X\cap_jY$ satisfy 
    $x_n\overset{n\rightarrow\infty}{\rightharpoonup}x\text{ in
    }X\cap_j Y$ if and only if
    $x_n\overset{n\rightarrow\infty}{\rightharpoonup}x\text{ in }X$
    and
    ${jx_n\overset{n\rightarrow\infty}{\rightharpoonup}jx}$  in $Y$.
  \item [\textrm (iv)] \textbf{Second characterization of weak
      convergence in $X\cap_j Y$:} In addition let $X$ be
    reflexive. A sequence $(x_n)_{n\in\mathbb{N}}\subset
    X\cap_jY$ and $x \in X\cap_jY$ satisfy 
    $x_n\overset{n\rightarrow\infty}{\rightharpoonup}x\text{ in
    }X\cap_j Y$ if and only if
   $\sup_{n\in\mathbb{N}}{\|x_n\|_X}<\infty$ and
    $jx_n\overset{n\rightarrow\infty}{\rightharpoonup}jx\text{ in }Y$.
  \end{description}
\end{prop}

\begin{proof}
  \textbf{ad (i)} Let $X\times Y$ be the product of $X$ and $Y$, which
  equipped with norm $\|(x,y)^\top\|_{X\times Y}:=\|x\|_X+\|y\|_Y$ is
  a Banach space. Hence,
  \begin{align*}
    L:X\cap_jY\rightarrow G(j)\subseteq X\times Y \colon x \mapsto (x,jx)^\top
  \end{align*}
  is a linear isometric isomorphism of $X\cap_jY$ onto
  $G(j)$. Therefore, $G(j)$ is closed in $X\times Y$.
	
  \textbf{ad (ii)}: If $X$ and $Y$ are reflexive, then also
  $X\times Y$ is reflexive. Since $G(j)$ is closed in $X\times Y$, it
  is also reflexive. As $L$ is an isomorphism we finally transfer the
  reflexivity from $G(j)$ to $X\cap_jY$. If $X$ and $Y$ are separable,
  then also $X\times Y$ is separable and $G(j)$ as well. As $L$ is an
  isometric isomorphism $X\cap_jY$ has to be separable as well.
	
  \textbf{ad (iii)}: Follows from the fact that weak convergence in
  $X\times Y$ is characterized by weak convergence of all components
  in conjunction with the isometric isomorphism $L$.

  \textbf{ad (iv)}: 
  The necessity follows immediately from (iii). To prove the
  sufficiency let  ${x \in X\cap_jY}$ and $(x_n)_{n\in\mathbb{N}} \subset X\cap_jY$
  satisfy $\sup_{n\in\mathbb{N}}{\|x_n\|_X}<\infty$ and
  $jx_n\overset{n\rightarrow\infty}{\rightharpoonup}jx\text{ in }Y$.
  Let $(x_n)_{n\in\Lambda}\subseteq X\cap_j Y$ with
  $\Lambda\subseteq \mathbb{N}$ be an arbitrary subsequence. In
  particular, $(x_n)_{n\in\Lambda}\subseteq X$ is bounded. Then
  Eberlein-\v{S}muljan's theorem yields the existence of both a
  subsequence $(x_n)_{n\in\Lambda_1}\subseteq X\cap_j Y$ with
  $\Lambda_1\subseteq \Lambda$ and an element $\tilde{x}\in X$ such
  that
  \begin{align*}
    x_n\overset{n\rightarrow\infty}{\rightharpoonup}\tilde{x} \quad \text{ in }X\;(n\in\Lambda_1).
  \end{align*}
  We have $((x_n,jx_n)^\top)_{n\in\Lambda_1}\subseteq G(j)$. As weak
  convergence of all components implies weak convergence in the
  corresponding Cartesian product we obtain
  \begin{align*}
    (x_n,jx_n)^\top\overset{n\rightarrow\infty}{\rightharpoonup}(\tilde{x},jx)^\top\quad
    \text{ in }X\times Y\;(n\in \Lambda_1). 
  \end{align*} 
  $G(j)$ is weakly closed in $X\times Y$, as it is convex and closed
  in $X\times Y$ (cf.~Proposition \ref{6.6}). In consequence, it holds
  $(\tilde{x},jx)^\top\in G(j)$, i.e., $\tilde{x}\in X\cap_j Y$ and
  $jx=j\tilde{x}$ in $Y$. From the injectivity of
  $j:X\cap_j Y\rightarrow Y$ we deduce further that $x=\tilde{x}$ in
  $X\cap_j Y$. Thus, the first characterization of weak convergence in
  pull-back intersections provides
  \begin{align*}
    x_n\overset{n\rightarrow\infty}{\rightharpoonup}x \quad \text{ in }X\cap_j Y\;(n\in \Lambda_1).
  \end{align*}
  Hence, $x\in X\cap_j Y$ is weak accumulation point of each
  subsequence of $(x_n)_{n\in\mathbb{N}}\subseteq X\cap_j Y$. The
  standard convergence principle (cf.~\cite[Kap.~I, Lemma 5.4]{GGZ74}) yields
  $x_n\overset{n\rightarrow\infty}{\rightharpoonup}x\text{ in }X\cap_j
  Y$.\hfill$\square$
\end{proof}

Now, we have created an appropriate framework to give a detailed description of the concept of pre-evolution triples.

\begin{defn}[Pre-evolution triple]\label{def:pre-evol}
  Let $(V,H):=(V,H,Z,e_V,e_H)$ be a compatible couple,
  $(V,\|\cdot\|_V)$ a separable, reflexive Banach space and
  $(H,(\cdot,\cdot)_H)$ a separable Hilbert space. In this situation
  the pull-back intersection of $V$ and $H$ is defined as
  $V\cap_jH:=e_V^{-1}\big (R(e_V) \cap R(e_H)\big )$, and the
  intersection embedding is defined as
  $j:=e_H^{-1} e_V: V\cap_jH \to H$. If
  \begin{align*}
    \overline{R(j)}^{\|\cdot\|_H}=\overline{j(V\cap_j H)}^{\|\cdot\|_H}=H,
  \end{align*}
  then the triple $(V,H,j)$ is said to be a \textbf{pre-evolution
    triple}.  Let $R:H\rightarrow H^*$ be the Riesz isomorphism with
  respect to $(\cdot,\cdot)_H$. As $j$ is a dense embedding the
  adjoint \mbox{$j^*:H^*\rightarrow (V\cap_j H)^*$} and therefore
  $e:=j^*Rj: V\cap _jH \rightarrow (V\cap _jH)^*$ are embeddings as
  well. We call $e$ the $\textbf{canonical\;embedding}$ of
  $(V,H,j)$. Note that 
  \begin{align}\label{eq:iden}
    \langle ev,w\rangle_{V\cap_j H}=(jv,jw)_H\quad \text{ for all }v,w\in
    V\cap_j H.
  \end{align} 
\end{defn}

\begin{rmk}
	The notion of a pre-evolution triple generalizes the standard notion
	of an evolution triple. Note that an evolution triple is a
	pre-evolution triple, since $(V,H,H,j,\text{id}_H)$ is a compatible
	couple. Moreover, the intersection embedding is the embedding $j$, and
	we have $V=V\cap _jH$ with norm equivalence
	$\|\cdot \|_{V} \sim \|\cdot \|_{V\cap _j H}$. Thus, if the
	pre-evolution triple is an evolution triple we can just replace the
	intersection $V\cap_j H$ by $V$. On the other hand if $(V,H,j)$ is a
	pre-evolution triple, then $(V\cap _j H, H,j)$ is an evolution triple.
\end{rmk}

\subsection{Bochner-Lebesgue spaces}
In this paragraph we collect some well known results concerning
Bochner-Lebesgue spaces, which will be used in the following. By
$(X,\|\cdot\|_X)$ and $(Y,\|\cdot\|_Y)$ we always denote Banach spaces, by $I:=\left(0,T\right)$, with $0<T<\infty$, a finite
time interval and by $\mathbfcal{M}(I,X)$ the vector space of Bochner measurable functions
from $I$ into $X$.

\begin{prop}\label{2.2}
  Let $\boldsymbol{x}:I\rightarrow X$ be a function such that there
  exists a sequence $(\boldsymbol{x}_n)_{n\in\mathbb{N}}\subseteq \mathbfcal{M}(I,X)$ with 
  \begin{align*}
    \boldsymbol{x}_n(t)\overset{n\rightarrow\infty}{\rightharpoonup}\boldsymbol{x}(t)\text{ in }X
  \end{align*}
	for almost every $t\in I$. Then $\boldsymbol{x}\in \mathbfcal{M}(I,X)$.
\end{prop}

\begin{proof}
  We apply Pettis' theorem (cf. \cite[Chapter V, Theorem: (Pettis)]{Yos80}) to obtain Lebesgue measurable sets
  $N_n\subseteq I$, $n\in\mathbb{N}$, such that $I\setminus N_n$ are null sets
  and $\boldsymbol{x}_n(N_n)$ are separable.  Thus, if we
  replace $X$ by the closure of
  $\text{span}\{\bigcup_{n\in\mathbb{N}}{\boldsymbol{x}_n(N_n)}\}$,
  it turns out that it suffices to treat the case of separable
  $X$. For a proof of the latter one we refer to \cite[Folgerung
  1.10]{Ru04}.\hfill$\square$
\end{proof}

\begin{prop}\label{2.3}
  Let $(X,\|\cdot\|_X)$ be a reflexive Banach space and $1< p<\infty$. If
  the sequence $(\boldsymbol{x}_n)_{n\in\mathbb{N}}\subseteq L^p(I,X)$ is bounded
  and satisfies 
  \begin{align*}
    \boldsymbol{x}_n(t)\overset{n\rightarrow\infty}{\rightharpoonup}\boldsymbol{x}(t)\text{ in }X
  \end{align*}
  for almost every $t\in I$, then
  $\boldsymbol{x}_n\overset{n\rightarrow\infty}{\rightharpoonup}\boldsymbol{x}\text{
    in }L^p(I,X).$
\end{prop}

\begin{proof} It suffices to treat the case
  $\boldsymbol{x}=\boldsymbol{0}$ in $L^p(I,X)$. For arbitrary
  $\boldsymbol{x}^*\in L^{p'}(I,X^*)\cong (L^p(I,X))^*$ we get 
  $\langle  \boldsymbol{x}^*(t),\boldsymbol{x}_n(t)\rangle_X\overset{n\rightarrow\infty}{\rightarrow}
  0$ for almost every $t\in I$. In particular, for Lebesgue measurable
  $E\subseteq I$ we obtain
  \begin{align*}
    \int_E{\vert\langle
    \boldsymbol{x}^*(s),\boldsymbol{x}_n(s)\rangle_X\vert\,ds}\leq
    \|\boldsymbol{x}^*\chi_E\|_{L^{p'}(I,X^*)}\|\boldsymbol{x}_n\|_{L^p(I,X)}\leq
    c\|\boldsymbol{x}^*\chi_E\|_{L^{p'}(I,X^*)}, 
  \end{align*}
  where we exploited the boundedness of
  $(\boldsymbol{x}_n)_{n\in\mathbb{N}}\subseteq L^p(I,X)$. Thus,
  $(\langle
  \boldsymbol{x}^*(\cdot),\boldsymbol{x}_n(\cdot)\rangle_X)_{n\in\mathbb{N}}\subseteq
  L^1(I)$ is uniformly integrable and Vitali's theorem in conjunction
  with the representation of the duality product in Bochner-Lebesgue
  spaces yields
  $\langle \boldsymbol{x}^*,\boldsymbol{x}_n\rangle_{L^p(I,X)}
  \overset{n\rightarrow\infty}{\rightarrow}0$.\hfill$\square$
\end{proof}

\begin{rmk}\label{rem:2.16}
  The converse implication in Proposition \ref{2.3} is in
  general not true. This can be seen by the following easy example. Let
  $I=\left(0,2\pi\right)$, $p\in \left(1,\infty\right)$,
  $X=\mathbb{R}$ and
  $(\boldsymbol{x}_n)_{n\in\mathbb{N}}\subseteq
  L^\infty(I,\mathbb{R})$, given via $\boldsymbol{x}_n(t)=\sin(nt)$
  for every $t\in I$ and all $n\in\mathbb{N}$. Then, there holds
  $\boldsymbol{x}_n\rightharpoonup\boldsymbol{0}$ in
  $L^p(I,\mathbb R)$ $(n\to\infty)$, but not
  $\boldsymbol{x}_n(t)\weakto 0$ for almost every $t\in I$ $(n\to\infty)$.
\end{rmk}

\begin{defn}[Induced operator]\label{ind}
	Let $A(t):X\to Y$, $t\in I$, be a family of operators with the following properties:
	\begin{description}
		\item[\textbf{(C.1)}]\hypertarget{C.1} $A(t):X\rightarrow Y$ is demi-continuous for
		almost every $t\in I$.
		\item[\textbf{(C.2)}]\hypertarget{C.2} $A(\cdot)x:I\rightarrow Y$ is Bochner measurable for
		all $x\in X$. 
	\end{description}
Then we define  the \textbf{induced operator} $\mathbfcal{A}:\mathbfcal{M}(I,X)\to\mathbfcal{M}(I,Y)$ of $\{A(t)\}_{t\in I}$ by
\begin{align*}
  (\mathbfcal{A}\boldsymbol{x})(t)=A(t)(\boldsymbol{x}(t))\quad\text{
  in }Y 
\end{align*}
for almost every $t\in I$ and all $\boldsymbol{x}\in\mathbfcal{M}(I,X)$.
\end{defn}

In the case $X=V$ and $Y=V^*$, where $(V,H,j)$ is an evolution triple,
the well-definedness of the induced operator is proved in
\cite{Zei90B}. The next lemma treats the general case. 

\begin{lem}\label{inda}
  Let $A(t):X\to Y$, $t\in I$, be a family of operators satisfying
  (\hyperlink{C.1}{C.1}) and (\hyperlink{C.2}{C.2}). Then the induced
  operator $\mathbfcal{A}:\mathbfcal{M}(I,X)\to\mathbfcal{M}(I,Y)$ is
  well-defined.
\end{lem}

\begin{proof}
  For $\boldsymbol{x}\in \mathbfcal{M}(I,X)$ there exists a sequence
  of simple functions
  $(\boldsymbol{s}_n)_{n\in\mathbb{N}}\subseteq \mathcal{S}(I,X)$,
  i.e., $\boldsymbol{s}_n(t)=\sum_{i=1}^{k_n}{s^n_i\chi_{E_i^n}(t)}$,
  where $s^n_i\in X$, $k_n\in\mathbb{N}$ and
  $E_i^n\in \mathcal{L}^1(I)$ with $E_i^n\cap E_j^n=\emptyset$ for
  $i\neq j$, which is converging almost everywhere to
  $\boldsymbol{x}(t)$ in $X$. Due to
  (\hyperlink{C.2}{C.2}) 
  $\mathbfcal{A}s_i^n=A(\cdot)s_i^n:I\rightarrow Y$ and therefore also
  \begin{align*}
    \mathbfcal{A}\boldsymbol{s}_n=\sum_{i=1}^{k_n}{\chi_{E_i^n}\mathbfcal{A}
    s_i^n+\chi_{\big\{I\setminus\bigcup_{i=1}^{k_n}{E_i^n}\big\}}\mathbfcal{A}0}:
    I\rightarrow Y
  \end{align*}
  are Bochner measurable and converge almost everywhere weakly to
  $(\mathbfcal{A}\boldsymbol{x})(t)$ in $Y$ due to
  (\hyperlink{C.1}{C.1}). Thus, Proposition \ref{2.2} ensures
  $\mathbfcal{A}\boldsymbol{x}\in \mathbfcal{M}(I,Y)$. \hfill$\square$
\end{proof}
The next proposition shows that linear and continuous operators
between Banach spaces transmit their properties to the induced
operator between Bochner-Lebesgue spaces.

\begin{prop}
  \label{2.4}
  Let  $1\leq p\leq \infty$ and let $A:X\rightarrow Y$ be linear and
  continuous. Then the induced operator
  $\mathbfcal{A}$ 
  is well-defined, linear and continuous as an operator from
  $L^p(I,X)$ into $L^p(I,Y)$. Furthermore, it
  holds:
  \begin{description}[{(iii)}]
  \item[{(i)}]
    $A\left(\int_I{\boldsymbol{x}(s)
       \, ds}\right)=\int_I{(\mathbfcal{A}\boldsymbol{x})(s)\,ds}\text{ in
    }Y$ for all $\boldsymbol{x}\in L^p(I,X)$.
  \item[{(ii)}] If $A:X\rightarrow Y$ is additionally
    injective, then $\mathbfcal{A}:L^p(I,X)\rightarrow L^p(I,Y)$ is
    injective as well. In particular, the inverse function
    $\mathbfcal{A}^{-1}:R(\mathbfcal{A})\rightarrow L^p(I,X)$ is
    well-defined and satisfies
    $(\mathbfcal{A}^{-1}\boldsymbol{y})(t)=A^{-1}(\boldsymbol{y}(t))$
    for almost every $t\in I$ and all
    $\boldsymbol{y}\in R(\mathbfcal{A})$.
  \item[{(iii)}] If $A:X\rightarrow Y$ is an isomorphism, then
    also $\mathbfcal{A}:L^p(I,X)\rightarrow L^p(I,Y)$ is an
    isomorphism. 
  \end{description}
\end{prop}

\begin{proof}
  Concerning the well-definedness, linearity and boundedness including
  point (i) we refer to \cite[Chapter V, 5. Bochner's Integral,
  Corollary 2]{Yos80}. The verification of assertions (ii) and (iii) is
  elementary and thus omitted.\hfill$\square$
\end{proof}

The next remark examines how the concept of pull-back intersections transfer to the Bochner-Lebesgue level.

\begin{rmk}[Induced compatible couple]\label{2.5}
  {\rm
    Let $(X,Y)=(X,Y,Z,e_X,e_Y)$ be a compatible couple (cf.~Definition
  \ref{6.2}) and $1\leq p,q\leq\infty$. In \cite[Chapter 3, 
  Theorem 1.3]{BS88} it is proved that the sum
  $R(e_X)+R(e_Y)\subseteq Z$ equipped with the norm
  \begin{align*}
    \|z\|_{R(e_X)+R(e_Y)}:=\inf_{\substack{x\in X,y\in Y\\z=e_Xx+e_Yy}}{\max\{\|x\|_X,\|y\|_Y\}}
  \end{align*}
  is a Banach space.  Then both $e_X:X\rightarrow R(e_X)+R(e_Y)$ and
  $e_Y:Y\rightarrow R(e_X)+R(e_Y)$ are embeddings (cf.~Definition
  \ref{6.1}) and therefore due to Proposition \ref{2.4} also the
  induced operators
  \begin{align*}
    &\boldsymbol{e}_X:L^p(I,X)\rightarrow L^1(I,R(e_X)+R(e_Y))
      \text{ given via }(\boldsymbol{e}_X\boldsymbol{x})(t)
      :=e_X(\boldsymbol{x}(t))\text{ for a.e. }t\in I,
    \\
    &\boldsymbol{e}_Y:L^q(I,Y)\rightarrow L^1(I,R(e_X)+R(e_Y))
      \text{ given via }(\boldsymbol{e}_Y\boldsymbol{y})(t)
      :=e_Y(\boldsymbol{y}(t))\text{ for a.e. }t\in I.
  \end{align*}
  Consequently, 
  \begin{align*}
    (L^p(I,X),L^q(I,Y))&=(L^p(I,X),L^q(I,Y),L^1(I,R(e_X)+R(e_Y)),
    \boldsymbol{e}_X,\boldsymbol{e}_Y)
  \end{align*}
  is a compatible couple. In accordance with Definition \ref{6.3},
  the pull-back intersection 
  \begin{align*}
    L^p(I,X)\cap_{\boldsymbol{j}}L^q(I,Y) 
  \end{align*}
  where $
  \boldsymbol{j}:=\boldsymbol{e}_Y^{-1}\boldsymbol{e}_X$, and the
  corresponding intersection embedding 
  \begin{align*}
    \boldsymbol{j}:L^p(I,X)\cap_{\boldsymbol{j}}L^q(I,Y)\rightarrow
    L^q(I,Y)
  \end{align*}
  is well-defined.}
\end{rmk}

Next we give an alternative representation of pull-back intersections
of Bochner-Lebesgue spaces, from which we are able to deduce Bochner
measurability with respect to $X\cap_j Y$ directly.

\begin{prop}
  \label{2.6}
  Let $(X,Y)$ be a compatible couple and $1\leq p\leq q\leq\infty$. Then
  \begin{align*}
    (L^p(I,X\cap_j Y),L^q(I,Y))=(L^p(I,X\cap_j Y),L^q(I,Y),L^p(I,Y),\boldsymbol{j},\text{id}_{L^q(I,Y)})
  \end{align*}
  is a compatible couple, where $\boldsymbol{j}$ is defined in 
  Remark \ref{2.5}. Thus,
  $L^p(I,X\cap_j Y)\cap_{\boldsymbol{j}}L^q(I,Y)$ and
  $\boldsymbol{j}:L^p(I,X\cap_j
  Y)\cap_{\boldsymbol{j}}L^q(I,Y)\rightarrow L^q(I,Y)$ are
  well-defined. In particular, it holds
  \begin{align*}
    L^p(I,X\cap_j Y)\cap_{\boldsymbol{j}}L^q(I,Y)= L^p(I,X)\cap_{\boldsymbol{j}}L^q(I,Y)
  \end{align*}
  with norm equivalence. 
\end{prop}

\begin{proof}
  As $j:X\cap_j Y\rightarrow Y$ is an embedding, the induced operator
  $\boldsymbol{j}:L^p(I,X\cap_j Y)\rightarrow L^p(I,Y)$ is an
  embedding as well, due to Proposition \ref{2.4}. Therefore, 
  \begin{align*}
    (L^p(I,X\cap_j Y),L^q(I,Y))=(L^p(I,X\cap_j Y),L^q(I,Y),L^p(I,Y),\boldsymbol{j},\text{id}_{L^q({I},Y)})
  \end{align*}
  is a compatible couple, and 
  $L^p(I,X\cap_j Y)\cap_{\boldsymbol{j}}L^q(I,Y)$ and
  $\boldsymbol{j}:L^p(I,X\cap_j
  Y)\cap_{\boldsymbol{j}}L^q(I,Y)\rightarrow L^q(I,Y)$ are
  well-defined. Proposition \ref{2.4} also implies
  $(\boldsymbol{j}^{-1}\boldsymbol{y})(t)=j^{-1}(\boldsymbol{y}(t))
  =e_X^{-1}e_Y(\boldsymbol{y}(t))=(\boldsymbol{e}_X^{-1}\boldsymbol{e}_Y
  \boldsymbol{y})(t)$  for almost every $t\in I$ and all
  $\boldsymbol{y}\in  R(\boldsymbol{j})=R(\boldsymbol{e}_Y^{-1}\boldsymbol{e}_X)$,
  i.e., $\boldsymbol{j}^{-1}=\boldsymbol{e}_X^{-1}\boldsymbol{e}_Y$ on
  $R(\boldsymbol{j})$. From the latter and Definition \ref{6.3} we
  obtain
  \begin{align*}
    L^p(I,X\cap_j Y)\cap_{\boldsymbol{j}}L^q(I,Y)
    &=\boldsymbol{j}^{-1}(R(\boldsymbol{j})\cap
      L^q(I,Y))\\&=\boldsymbol{e}_X^{-1}\boldsymbol{e}_Y(R(\boldsymbol{j})\cap
    L^q(I,Y))\\&=\boldsymbol{e}_X^{-1}(R(\boldsymbol{e}_Y\boldsymbol{j})\cap
    R(\boldsymbol{e}_Y\text{id}_{L^q(I,Y)})) 
    \\
    &=\boldsymbol{e}_X^{-1}(R(\boldsymbol{e}_X)\cap R(\boldsymbol{e}_Y\text{id}_{L^q(I,Y)}))
    \\
    &=L^p(I,X)\cap_{\boldsymbol{j}}L^q(I,Y).
  \end{align*}
  The verification of the stated norm equivalence is an elementary
  calculation and thus omitted.\hfill$\square$
\end{proof}

\subsection{Evolution equations}
For a pre-evolution triple $(V,H,j)$, $I:=\left(0,T\right)$, with $0<T<\infty$, and $1<p< \infty$ we set
\[
  \mathbfcal{X}:=L^p(I,V\cap _jH), \qquad 
 \mathbfcal{Y}:=L^\infty(I,H).
\]

\begin{defn}\label{2.10}
  Let $(V,H,j)$ be a pre-evolution triple and $1< p<\infty$. A
  function $\boldsymbol{x}\in\mathbfcal{X}$ possesses a
  \textbf{generalized time derivative with respect to the canonical
    embedding $e$ of $(V,H,j)$} if there exists a function
  $\boldsymbol{w}\in \mathbfcal{X}^*$ such that
  \begin{align*}
	-\int_I{(j(\boldsymbol{x}(s)),jv)_H\varphi\prime(s)\,ds}=
    \int_I{\langle\boldsymbol{w}(s),v\rangle_{V\cap_jH}\varphi(s)\,ds} 
  \end{align*}
  for all $v\in V\cap_jH$ and $\varphi\in C_0^\infty(I)$. As this function $\boldsymbol{w}\in \mathbfcal{X}^*$ is unique (cf. \cite[Proposition 23.18]{Zei90A}),
  $\frac{d_e\boldsymbol{x}}{dt}:=\boldsymbol{w}$ is well-defined. By 
  \begin{align*}
  	\mathbfcal{W}:=W^{1,p,p'}_e(I,V\cap_jH,(V\cap_jH)^*):=\Big\{\boldsymbol{x}\in
    L^p(I,V\cap_j H)\,\Big|\, \exists \, \frac{d_e\boldsymbol{x}}{dt}\in
    L^{p'}(I,(V\cap_j H)^*)\Big\} 
  \end{align*}
  we denote the \textbf{Bochner-Sobolev space with respect to $e$}.
\end{defn}

\begin{prop}\label{2.11}
  Let $(V,H,j)$ be a pre-evolution triple and $1<p<\infty$. Then it holds:
	\begin{description}[(iii)]
		\item[(i)] The space $\mathbfcal{W}$ forms a Banach space equipped with the norm 
		\begin{align*}
		\|\cdot\|_{\mathbfcal{W}}:=\|\cdot\|_{\mathbfcal{X}}+\left\|\frac{d_e}{dt}\,\cdot\,\right\|_{\mathbfcal{X}^*}.
		\end{align*}
        \item[(ii)] Given $\boldsymbol{x}\in \mathbfcal{W}$ the
          function $\boldsymbol{j}\boldsymbol{x}\in L^p(I,H)$, given
          via
          $(\boldsymbol{j}\boldsymbol{x})(t):=j(\boldsymbol{x}(t))$
          for almost every $t\in I$, possesses a unique representation
          in $C^0(\overline{I},H)$ and the resulting mapping
          $\boldsymbol{j}:\mathbfcal{W}\rightarrow
          C^0(\overline{I},H)$ is an embedding.
        \item[(iii)] \textbf{Generalized integration by parts formula:} It holds
          \begin{align*}
            \int_{t'}^{t}{\left\langle
            \frac{d_e\boldsymbol{x}}{dt}(s),\boldsymbol{y}(s)\right\rangle_{V\cap_jH}\,ds}
            =\big[((\boldsymbol{j}\boldsymbol{x})(s), (\boldsymbol{j}
            \boldsymbol{y})(s))_H\big]^{s=t}_{s=t'}-\int_{t'}^{t}{\left\langle
            \frac{d_e\boldsymbol{y}}{dt}(s),\boldsymbol{x}(s)\right\rangle_{V\cap_jH}\,ds} 
          \end{align*}
          for all $\boldsymbol{x},\boldsymbol{y}\in \mathbfcal{W}$ and
          $t,t'\in \overline{I}$ with $t'\leq t$.
	\end{description}  
\end{prop}

\begin{proof}
  A straightforward adaption of \cite[Proposition 23.23]{Zei90A},
  since $(V\cap _j H, H,j)$ is an evolution triple.\hfill $\square$
\end{proof}

\begin{defn}[Evolution equation]\label{2.12}
  Let $(V,H,j)$ be a pre-evolution triple and $1<p< \infty$. Moreover, let
  $\boldsymbol{y}_0\in H$ be an initial value,
  $\boldsymbol{f}\in\mathbfcal{X}^*$ a right-hand side and
  $\mathbfcal{A}:\mathbfcal{X}\cap_{\boldsymbol{j}}\mathbfcal{Y}\rightarrow\mathbfcal{X}^*$
  an operator. Then the initial value problem
  \begin{align}
    \begin{aligned}
      \frac{d_e\boldsymbol{y}}{dt}+\mathbfcal{A}\boldsymbol{y}&=\boldsymbol{f}\quad&&\text{ in }\mathbfcal{X}^*,\\
      (\boldsymbol{j}\boldsymbol{y})(0)&=\boldsymbol{y}_0&&\text{ in }H
    \end{aligned}\label{eq:1}
  \end{align}
  is said to be an \textbf{evolution equation}. The initial condition
  has to be understood in the sense of the unique continuous
  representation $\boldsymbol{j}\boldsymbol{y}\in C^0(\overline{I},H)$
  (cf. Proposition \ref{2.11} (ii)).
\end{defn}

\section{Bochner pseudo-monotonicity and Bochner coercivity}
\label{sec:3}
 In this section we introduce the notions Bochner pseudo-monotonicity and Bochner coercivity.

\begin{defn}[Bochner pseudo-monotonicity]\label{3.2}
  Let $(V,H,j)$ be a pre-evolution triple and $1<p<\infty$. An
  operator
  $\mathbfcal{A}:D(\bfA)\subseteq \mathbfcal{X}\rightarrow\mathbfcal{X}^*$ with 
  $D(\bfA) \supseteq \mathbfcal{X}\cap_{\boldsymbol{j}}\mathbfcal{Y}$
  is said to be \textbf{Bochner pseudo-monotone} if for a sequence
  $(\boldsymbol{x}_n)_{n\in\mathbb{N}}\subseteq
  \mathbfcal{X}\cap_{\boldsymbol{j}}\mathbfcal{Y}$ from
  \begin{alignat}{2}
    \boldsymbol{x}_n\overset{n\rightarrow\infty}{\rightharpoonup}
    &\boldsymbol{x}\quad &&\text{ in }\mathbfcal{X}\label{eq:3.3},
    \\
    \boldsymbol{j}\boldsymbol{x}_n\;\;\overset{\ast}{\rightharpoondown}\;\;
    &\boldsymbol{j}\boldsymbol{x}&&\text{ in }\mathbfcal{Y} \quad
    (n\rightarrow\infty),
    \label{eq:3.3a}
    \\
    (\boldsymbol{j}\boldsymbol{x}_n)(t)\overset{n\rightarrow\infty}{\rightharpoonup}
    &(\boldsymbol{j}\boldsymbol{x})(t)\quad &&\text{ in }H\text{ for
      a.e. }t\in I,\label{eq:3.4}
  \end{alignat}
  and
  \begin{align}
    \limsup_{n\rightarrow\infty}{\langle \mathbfcal{A} \boldsymbol{x}_n,
    \boldsymbol{x}_n-\boldsymbol{x}\rangle_{\mathbfcal{X}}}
    \leq 0\label{eq:3.5}
  \end{align}
  it follows that
  $ \langle
  \mathbfcal{A}\boldsymbol{x},\boldsymbol{x}-\boldsymbol{y}\rangle_{\mathbfcal{X}}\leq
  \liminf_{n\rightarrow\infty}{\langle
    \mathbfcal{A}\boldsymbol{x}_n,\boldsymbol{x}_n-\boldsymbol{y}\rangle_{\mathbfcal{X}}}$
  for all $\boldsymbol{y}\in \mathbfcal{X}$.
\end{defn}

We will see in the proof of Theorem \ref{4.1} that
\eqref{eq:3.3}--\eqref{eq:3.5} are natural properties of a sequence
$(\boldsymbol {x}_n)_{n \in \mathbb N}\subseteq
\mathbfcal{X}\cap_{\boldsymbol{j}}\mathbfcal{Y}$ coming from an
appropriate Galerkin approximation of \eqref{eq:1}, if $\bfA$
satisfies appropriate additional assumptions. In fact,
\eqref{eq:3.3} usually is a consequence of the coercivity of
$\mathbfcal{A}$, \eqref{eq:3.3a} stems from the time derivative, while
\eqref{eq:3.4} and \eqref{eq:3.5} follow directly from the Galerkin
approximation.

\begin{rmk}\label{comp2}
  One easily sees, that each pseudo-monotone operator
  $\mathbfcal{A}: \mathbfcal{X} \rightarrow\mathbfcal{X}^*$ is Bochner
  pseudo-monotone  with $D(\bfA)= \mathbfcal{X}$.
  Note that converse is not true in
  general. In fact, there exist Bochner pseudo-monotone operators
  which are not pseudo-monotone. This can be seen by the following
  example.
	
  Let $I=\left(0,T\right)$, with $0<T<\infty$, 
  $V=W^{1,p}_{0,\text{div}}(\Omega)$, $p>3$, $H=L^2_{\divo}(\Omega)$,
  $\mathbfcal{X}=L^p(I,V)$, $\mathbfcal{Y}=L^\infty(I,H)$ and let 
  $B:V\to V^*$ be the convective term, defined through
  $\langle B\bv,\bw\rangle_V=\int_{\Omega}{\bv\otimes
    \bv:\bD\bw\,dx}$ for all $\bv,\bw\in V$.  Then, $B:V\to V^*$ is
  compact, and thus pseudo-monotone. The unsteady convective term
  $\mathbfcal{B}:\mathbfcal{X}\to \mathbfcal{X}^*$, given via
  $(\mathbfcal{B}\boldsymbol{x})(t):=B(\boldsymbol{x}(t))$ in $V^*$
  for almost every $t\in I$ and all $\boldsymbol{x}\in \mathbfcal{X}$,
  is well-defined but neither compact nor pseudo-monotone. In fact, let
  $\bv,\bw\in V$ be fixed with $\langle B\bv,\bw\rangle_V<0$ and
  $(\boldsymbol{x}_n)_{n\in\mathbb{N}}\subseteq \mathbfcal{X} $, given
  via $\boldsymbol{x}_n(t)=\sin(nt)\bv$ for every 
  $t\in I$ and $n\in\mathbb{N}$. Then
  $\boldsymbol{x}_n\rightharpoonup \boldsymbol{0}$ in $\mathbfcal{X}$ $(n\to\infty)$
  and $\limsup_{n\to \infty} \langle \mathbfcal{B} \boldsymbol{x}_ n,
  \boldsymbol{x}_n-\boldsymbol{0}\rangle_{\mathbfcal{X}}= 0$, since
  $\langle \mathbfcal B\boldsymbol x_n, \boldsymbol x_n\rangle_{\mathbfcal
    X}=0$. For
  $\boldsymbol{y}(\cdot):=\bw\in \mathbfcal{X}$ it holds
  $ \liminf_{n\rightarrow\infty}{\langle
    \mathbfcal{B}\boldsymbol{x}_n,\boldsymbol{x}_n-\boldsymbol{y}\rangle_{\mathbfcal{X}}}=\pi
  \langle B\bv,\bw\rangle_V<0=\langle
  \mathbfcal{B}\boldsymbol{0},\boldsymbol{0}-\boldsymbol{y}\rangle_{\mathbfcal{X}}$,
  i.e., $\mathbfcal B$ is not pseudo-monotone, and thus also not compact. We
  emphasize that even
  $\boldsymbol{x}_n\overset{\ast}{\rightharpoondown}\boldsymbol{0}$ in
  $L^\infty(I,V)$ $(n\to\infty)$, but $\boldsymbol{x}_n(t)\rightharpoonup \boldsymbol{0}$ in $H$ for
  almost every $t\in I$ $(n\to\infty)$ is not satisfied, which is the
  additional requirement in the Bochner pseudo-monotonicity. 
	
  On the other hand, $\mathbfcal{B}:\mathbfcal{X}\to \mathbfcal{X}^*$ is Bochner
  pseudo-monotone. If $(\boldsymbol{x}_n)_{n\in\mathbb{N}}\subseteq
  \mathbfcal{X}\cap_{\boldsymbol{j}} 
  \mathbfcal{Y}$ satisfies \eqref{eq:3.3}--\eqref{eq:3.5}, then we
  infer from Landes' and Mustonen's compactness principle
  (cf. \cite[Propositon 1]{LM87}) that
  $\boldsymbol{x}_n\to \boldsymbol{x}$ almost everywhere in
  $I\times \Omega$ $(n\to\infty)$. As
  $\mathbfcal{X}\cap _{\boldsymbol{j}}\mathbfcal{Y}\hookrightarrow L^\rho(I\times
  \Omega)$, where $\rho=\frac{5}{3}\,p>2p'$, we thus gain
  $\boldsymbol{x}_n\otimes\boldsymbol{x}_n\to
  \boldsymbol{x}\otimes\boldsymbol{x}$ in $L^{p'}(I\times \Omega)$
  $(n\to \infty)$, which in turn implies
  $\mathbfcal{B}\boldsymbol{x}_n\to \mathbfcal{B}\boldsymbol{x}$ in
  $\mathbfcal{X}^*$ and therefore
  $\langle
  \mathbfcal{B}\boldsymbol{x},\boldsymbol{x}-\boldsymbol{y}\rangle_{\mathbfcal{X}}\leq
  \liminf_{n\rightarrow\infty}{\langle
    \mathbfcal{B}\boldsymbol{x}_n,\boldsymbol{x}_n-\boldsymbol{y}\rangle_{\mathbfcal{X}}}$
  for any $\boldsymbol{y}\in \mathbfcal{X}$.

  The convective term is for $p> \frac{11}{5}$ well-defined as an
  operator
  $\mathbfcal{B}:\mathbfcal{X}\cap _{\boldsymbol{j}}\mathbfcal{Y}\to \mathbfcal{X}^*$,
  as can be easily checked by H\"older's inequality.  Note that the
  above argumentation works also for $p > \frac {11}5$, since
  $\rho=\frac{5}{3}\,p>2p'$. Thus,
  $\mathbfcal{B}: D(\bfA)\subseteq \mathbfcal{X}\to \mathbfcal{X}^*$
  with $D(\bfA)= \mathbfcal{X}\cap _{\boldsymbol{j}}\mathbfcal{Y}$
  is Bochner pseudo-monotone for $p > \frac {11}5$.
\end{rmk}

Pseudo-monotonicity possesses two essential properties. On
the one hand, if $A:X\to X^*$ is pseudo-monotone,
$(x_n)_{n\in\mathbb{N}}\subseteq X$ satisfies
\eqref{eq:1.1}--\eqref{eq:1.2} and
$(Ax_n)_{n\in\mathbb{N}}\subseteq X^*$ is bounded, then
$Ax_n\overset{n\to\infty}{\rightharpoonup} Ax$ in $X^*$. On the other
hand, pseudo-monotonicity is stable under summation, in the sense that
the sum of two pseudo-monotone operators is pseudo-monotone again. We
will see that Bochner pseudo-monotonicity also possesses these two
essential properties.

\begin{prop}\label{3.6}
  Let $(V,H,j)$ be a pre-evolution triple, $1<p<\infty$ and
  $\mathbfcal{A}:D(\bfA) \subseteq
  \mathbfcal{X}\rightarrow\mathbfcal{X}^*$ be 
  Bochner pseudo-monotone. Then it holds:
  \begin{description}[(ii)]
  \item[(i)] If a sequence 
    $(\boldsymbol{x}_n)_{n\in\mathbb{N}}\subseteq
    \mathbfcal{X}\cap_{\boldsymbol{j}}\mathbfcal{Y}$ and
    $\boldsymbol{x}\in \mathbfcal{X}\cap_{\boldsymbol{j}}\mathbfcal{Y}$ satisfy
    \eqref{eq:3.3}--\eqref{eq:3.5}, and if 
    $(\mathbfcal{A}\boldsymbol{x}_n)_{n\in\mathbb{N}}\subseteq
    \mathbfcal{X}^*$ is bounded, then
    $\mathbfcal{A}\boldsymbol{x}_n\overset{n\rightarrow\infty}{\rightharpoonup}\mathbfcal{A}\boldsymbol{x}$
    in $\mathbfcal{X}^*$.
  \item[(ii)] If
    $\mathbfcal{B}:D(\bfB) \subseteq \mathbfcal{X}\rightarrow\mathbfcal{X}^*$
    is Bochner pseudo-monotone, then
    $\mathbfcal{A}+\mathbfcal{B}:D(\bfA)\cap D(\bfB)\rightarrow\mathbfcal{X}^*$
    is Bochner pseudo-monotone.
  \end{description}
\end{prop}

		

\begin{proof}
  \textbf{ad (i)} Since
  $\mathbfcal{X}^*$ is reflexive, we obtain a subsequence
  $(\mathbfcal{A}\boldsymbol{x}_n)_{n\in\Lambda}\subseteq
  \mathbfcal{X}^*$ with $\Lambda\subseteq\mathbb{N}$ and
  $\boldsymbol{\xi}\in \mathbfcal{X}^*$ such that
  $\mathbfcal{A}\boldsymbol{x}_n\overset{n\rightarrow\infty}{\rightharpoonup}\boldsymbol{\xi}$
  in $\mathbfcal{X}^*$ $(n\in\Lambda)$. This, the Bochner
  pseudo-monotonicity of
  $\mathbfcal{A}:D(\bfA)\subseteq \mathbfcal{X}\rightarrow\mathbfcal{X}^*$
  and \eqref{eq:3.5} imply
  \begin{align*}
    \langle\mathbfcal{A}\boldsymbol{x},\boldsymbol{x}-
    \boldsymbol{y}\rangle_{\mathbfcal{X}}
    &\leq \liminf_{\substack{n\rightarrow\infty\\n\in\Lambda}}
    {\langle\mathbfcal{A}\boldsymbol{x}_n,\boldsymbol{x}_n
    -\boldsymbol{y}\rangle_{\mathbfcal{X}}}
    \\
    &\leq \limsup_{\substack{n\rightarrow\infty\\n\in\Lambda}}
    {\langle\mathbfcal{A}\boldsymbol{x}_n,\boldsymbol{x}_n
    -\boldsymbol{x}\rangle_{\mathbfcal{X}}}+
    \limsup_{\substack{n\rightarrow\infty\\n\in\Lambda}}
    {\langle\mathbfcal{A}\boldsymbol{x}_n,\boldsymbol{x}
    -\boldsymbol{y}\rangle_{\mathbfcal{X}}}\leq
    \langle\boldsymbol{\xi},\boldsymbol{x}-\boldsymbol{y}\rangle_{\mathbfcal{X}} 
  \end{align*}
  for all $\boldsymbol{y}\in \mathbfcal{X}$ and therefore
  $\mathbfcal{A}\boldsymbol{x}=\boldsymbol{\xi}$ in
  $\mathbfcal{X}^*$. As this argumentation stays valid for each
  subsequence of
  $(\mathbfcal{A}\boldsymbol{x}_n)_{n\in\mathbb{N}}\subseteq
  \mathbfcal{X}^*$, $\mathbfcal{A}\boldsymbol{x}\in \mathbfcal{X}^*$
  is a weak accumulation point of each subsequence of
  $(\mathbfcal{A}\boldsymbol{x}_n)_{n\in\mathbb{N}} $. 
  Thus, the standard convergence principle (cf.~\cite[Kap.~I, Lemma 5.4]{GGZ74}) yields
  the assertion.
	
  \textbf{ad (ii)} Let $(\boldsymbol{x}_n)_{n\in\mathbb{N}}\subseteq
  \mathbfcal{X}\cap_{\boldsymbol{j}}\mathbfcal{Y}$ satisfy
  \eqref{eq:3.3}--\eqref{eq:3.4} and
  $\limsup_{n\to\infty}{\langle(\mathbfcal{A}+\mathbfcal{B})\boldsymbol{x}_n,\boldsymbol{x}_n
    -\boldsymbol{x}\rangle_{\mathbfcal{X}}}\leq
  0$.  Set $a_n:=\langle\mathbfcal{A}\boldsymbol{x}_n,\boldsymbol{x}_n
  -\boldsymbol{x}\rangle_{\mathbfcal{X}}$ and
  $b_n:=\langle\mathbfcal{B}\boldsymbol{x}_n,\boldsymbol{x}_n
  -\boldsymbol{x}\rangle_{\mathbfcal{X}}$ for
  $n\in\mathbb{N}$.  Then, it holds $\limsup_{n\to\infty}{a_n}\leq
  0$ and $\limsup_{n\to\infty}{b_n}\leq
  0$. In fact, suppose on the contrary, e.g.~that
  ${\limsup_{n\to\infty}{a_n}=a>
  0}$. Then, there exists a subsequence such that
  $a_{n_k}\to
  a$ $(k\to\infty)$, and therefore $\limsup_{k\to\infty}{b_{n_k}}\leq
  \limsup_{k\to\infty}{a_{n_k}+b_{n_k}}-\lim_{k\to\infty}{a_{n_k}}\leq-a<0$,
  i.e., a contradiction, since then the Bochner pseudo-monotonicity of
  $\mathbfcal{B}:D(\bfB)\subseteq \mathbfcal{X}\rightarrow\mathbfcal{X}^*$
  implies $0\leq \liminf_{n\to\infty,n\in\Lambda}{b_n}<
  0$. Thus, we have $\limsup_{n\to\infty}{a_n}\leq
  0$ and $\limsup_{n\to\infty}{b_n}\leq
  0$, and the Bochner pseudo-monotonicity of the operators 
  $\mathbfcal{A}$ and $\mathbfcal{B}$
  provides $\langle\mathbfcal{A}\boldsymbol{x},\boldsymbol{x}
  -\boldsymbol{y}\rangle_{\mathbfcal{X}}\leq
  \liminf_{n\to\infty}{\langle\mathbfcal{A}\boldsymbol{x}_n,\boldsymbol{x}_n
    -\boldsymbol{y}\rangle_{\mathbfcal{X}}}$ and
  $\langle\mathbfcal{B}\boldsymbol{x},\boldsymbol{x}
  -\boldsymbol{y}\rangle_{\mathbfcal{X}}\leq
  \liminf_{n\to\infty}{\langle\mathbfcal{B}\boldsymbol{x}_n,\boldsymbol{x}_n
    -\boldsymbol{y}\rangle_{\mathbfcal{X}}}$. Summing these
  inequalities yields the assertion.  \hfill$\square$
\end{proof}

\begin{defn}[Bochner coercivity]\label{3.20}
  Let $(V,H,j)$ be a pre-evolution triple and $1<p<\infty$. An
  operator
  $\mathbfcal{A}: D(\bfA) \subseteq \mathbfcal{X} \rightarrow\mathbfcal{X}^*$
  with  $D(\bfA) \supseteq \mathbfcal{X}\cap_{\boldsymbol{j}}\mathbfcal{Y}$
  is said to be
	\begin{description}[(ii)]
		\item[(i) \textbf{Bochner coercive with respect to
			$\boldsymbol{f}\in\mathbfcal{X}^*$ and $\boldsymbol{x}_0\in H$}]if there
		exists a constant \linebreak\mbox{$M:=M(\boldsymbol{f},\boldsymbol{x}_0,\mathbfcal{A})>0$} such
		that for all
		$\boldsymbol{x}\in\mathbfcal{X}\cap_{\boldsymbol{j}}\mathbfcal{Y}$
		from
		\begin{align}
		\frac{1}{2}\|(\boldsymbol{j}\boldsymbol{x})(t)\|_H^2
		+\int_{0}^{t}{\langle(\mathbfcal{A}\boldsymbol{x})(s)
			-\boldsymbol{f}(s),\boldsymbol{x}(s)\rangle_{V\cap_j H}\,ds}
		\leq \frac{1}{2}\|\boldsymbol{x}_0\|_H^2\quad\text{ for a.e. }t\in I\label{eq:bcoer}
		\end{align}
		it follows that $\|\boldsymbol{x}\|_{\mathbfcal{X}\cap_{\boldsymbol{j}}\mathbfcal{Y}}\leq M$.
		\item[(ii) \textbf{Bochner coercive}]if it is Bochner coercive with
		respect to $\boldsymbol{f}$ and $\boldsymbol{x}_0$ for all
		$\boldsymbol{f}\in\mathbfcal{X}^*$ and $\boldsymbol{x}_0\in H$.
	\end{description}
\end{defn}

Note that Bochner coercivity, similar to semi-coercivity
(cf.~\cite{Rou05}) in conjunction with Gronwall's inequality, takes
into account the information from the operator and the time
derivative. In fact, Bochner coercivity is a more general property. In
the context of the main theorem on pseudo-monotone perturbations of
maximal monotone mapping (cf.~\cite[§32.4.]{Zei90B}), which implies
Theorem \ref{thm:b}, Bochner coercivity is phrased in the spirit of a
local coercivity\footnote{$A:D(A)\subseteq V\to V^*$ is said to be
  coercive (cf.~\cite[§32.4.]{Zei90B}) with respect to $f\in V^*$, if
  $D(A)$ is unbounded and there exists a constant $R>0$, such that for
  $v\in V$ from $\langle Av,v\rangle_V\leq \langle f,v\rangle_V$ it
  follows $\|v\|_V\leq R$, i.e., all elements whose images with
  respect to $A$ do not grow beyond the data $f$ in this weak sense
  are contained in a fixed ball in $V$.} type condition of
$\frac{d_e}{dt}+\mathbfcal{A}:\mathbfcal{W}\subseteq \mathbfcal{X}\to
\mathbfcal{X}^*$. Being more precise, if
$\mathbfcal{A}:D(\mathbfcal{A})\subseteq \mathbfcal{X}\to
\mathbfcal{X}^*$ is Bochner coercive with respect to
$\boldsymbol{f}\in\mathbfcal{X}^*$ and $\boldsymbol{x}_0\in H$, then
for $\boldsymbol{x}\in \mathbfcal{W}$ from
${\|(\boldsymbol{j}\boldsymbol{x})(0)\|_H\leq
  \|\boldsymbol{x}_0\|_H}$, i.e., 
$\langle
\frac{d_e\boldsymbol{x}}{dt},\boldsymbol{x}\rangle_{\mathbfcal{X}}\ge
-\frac{1}{2}\|\boldsymbol{x}_0\|_H^2$, and
\begin{align}
	\bigg\langle \frac{d_e\boldsymbol{x}}{dt}+\mathbfcal{A}\boldsymbol{x},\boldsymbol{x}\chi_{\left[0,t\right]}\bigg\rangle_{\mathbfcal{X}}\leq \langle \boldsymbol{f},\boldsymbol{x}\chi_{\left[0,t\right]}\rangle_{\mathbfcal{X}}\quad\text{ for a.e. }t\in I.\label{eq:bcoer2}
\end{align}
it follows $\|\boldsymbol{x}\|_{\mathbfcal{X}\cap_{\boldsymbol{j}}\mathbfcal{Y}}\leq M$, since \eqref{eq:bcoer2} is just \eqref{eq:bcoer}. In other words, if the image of $\boldsymbol{x}\in\mathbfcal{W}$ with respect to $\frac{d_e}{dt}$ and $\mathbfcal{A}$ is bounded by the data $\boldsymbol{x}_0$, $\boldsymbol{f}$ in this weak sense, then $\boldsymbol{x}$ is contained in a fixed ball in $\mathbfcal{X}\cap_{\boldsymbol{j}}\mathbfcal{Y}$. We chose \eqref{eq:bcoer} instead of \eqref{eq:bcoer2} in Definition \ref{3.20}, since $\boldsymbol{x}\in \mathbfcal{X}\cap_{\boldsymbol{j}}\mathbfcal{Y}$ is not admissible in \eqref{eq:bcoer2}. 

Apart from that, there is a relation between Bochner coercivity and coercivity in the sense of Definition \ref{2.1}. In fact, in the case of
bounded operators
$\mathbfcal{A}:\mathbfcal{X}\rightarrow \mathbfcal{X}^*$, Bochner
coercivity extends the standard concept of coercivity.

\begin{lem}\label{3.21}
  Let $(V,H,j)$ be a pre-evolution triple and $1<p<\infty$. If the
  operator $\mathbfcal{A}:  D(\mathbfcal {A}) \subseteq \mathbfcal{X}
  \rightarrow \mathbfcal{X}^*$ with
  $D(\mathbfcal A)=\mathbfcal{X} \cap_{\boldsymbol{j}} \mathbfcal{Y}$
  is coercive and bounded, then $\mathbfcal{A}$ is Bochner coercive.
\end{lem}

\begin{proof}
  It suffices to show that
  $\mathbfcal{A}: D(\mathbfcal {A}) \subseteq \mathbfcal{X}
  \rightarrow\mathbfcal{X}^* $ is Bochner coercive with respect to $\boldsymbol{0}\in\mathbfcal{X}^*$ and $\boldsymbol{x}_0\in H$. For
  $\boldsymbol{f}\in \mathbfcal{X}^*\setminus\{\boldsymbol{0}\}$, we
  consider the shifted operator
  $\widehat{\mathbfcal{A}}:=\mathbfcal{A}-\boldsymbol{f}:D(\mathbfcal{A})\subset
  \mathbfcal{X}\rightarrow \mathbfcal{X}^*$ which is still coercive and
  bounded. Therefore, $\widehat{\mathbfcal{A}} $ is Bochner coercive
  with respect to $\boldsymbol{0}\in\mathbfcal{X}^*$ and $\boldsymbol{x}_0\in H$, and we conclude
  that $\mathbfcal{A}$ is Bochner coercive. To show that
  $\mathbfcal{A}:\mathbfcal{X}\cap_{\boldsymbol{j}}\mathbfcal{Y}
  \rightarrow\mathbfcal{X}^*$ is Bochner coercive with respect to
  $\boldsymbol 0$ and $\boldsymbol{x}_0$, we assume
  that
  $\boldsymbol{x}\in\mathbfcal{X}\cap_{\boldsymbol{j}}\mathbfcal{Y}$
  satisfies for almost every $t\in I$
  \begin{align}
    \frac{1}{2}{\|(\boldsymbol{j}\boldsymbol{x})(t)\|_H^2}
    +\langle\mathbfcal{A}\boldsymbol{x},\boldsymbol{x}
    \chi_{\left[0,t\right]}\rangle_{\mathbfcal{X}}\leq \frac{1}{2}\|\boldsymbol{x}_0\|_H^2.\label{eq:3.23}
  \end{align}
  Since
  $\mathbfcal{A}\colon D(\mathbfcal A) \subset \mathbfcal X \to
  \mathbfcal X^*$ is coercive there exists a constant
  $R:=R(\mathbfcal{A})>0$ such that
  $\langle
  \mathbfcal{A}\boldsymbol{w},\boldsymbol{w}\rangle_{\mathbfcal{X}}\ge
  \|\boldsymbol{w}\|_{\mathbfcal{X}}$ for all
  $\boldsymbol{w}\in D(\bfA) \supseteq \mathbfcal{X}\cap_{\boldsymbol{j}}\mathbfcal{Y}$
  such that $\|\boldsymbol{w}\|_\mathbfcal{X}\ge R$.  Next, we define
  $M_0:=\max\{R,\frac{1}{2}\|\boldsymbol{x}_0\|_H^2\}>0$ and suppose that
  $\|\boldsymbol{x}\|_\mathbfcal{X}>M_0\ge R$. Therefore, using the
  coercivity and \eqref{eq:3.23}, we conclude
  $M_0 <\|\boldsymbol{x}\|_\mathbfcal{X}\le \langle
  \mathbfcal{A}\boldsymbol{x},\boldsymbol{x}\rangle_\mathbfcal{X} \le
  \frac{1}{2}\|\boldsymbol{x}_0\|_H^2 \le M_0$, which is a contradiction. Thus,
  $\|\boldsymbol{x}\|_\mathbfcal{X}\leq M_0$ has to be valid. As
  $\mathbfcal{A}\colon D(\mathbfcal A) \subset \mathbfcal X \to
  \mathbfcal X^*$ is bounded there exists a constant
  $\Lambda:=\Lambda(M_0)>0$ such that
  $\|\mathbfcal A\boldsymbol{w}\|_{\mathbfcal{X}^*}\leq \Lambda$ for
  all
  $\boldsymbol{w}\in \mathbfcal{X}\cap_{\boldsymbol{j}}\mathbfcal{Y}$
  with $\|\boldsymbol{w}\|_{\mathbfcal{X}}\leq M_0$. This and
  \eqref{eq:3.23} imply
  $\|\boldsymbol{j}\boldsymbol{x}\|_{\mathbfcal{Y}}^2\leq \|\boldsymbol{x}_0\|_H^2+2\Lambda
  M_0$, which yields
  $\|\boldsymbol{x}\|_{\mathbfcal{X}\cap_{\boldsymbol{j}}\mathbfcal{Y}}\leq
  M_0+(\|\boldsymbol{x}_0\|_H^2+2\Lambda
  	M_0)^{1/2}=:M$.\hfill$\square$
\end{proof}

The following proposition provides sufficient conditions on a time-dependent family of operators such that the induced operator is bounded, Bochner pseudo-monotone and Bochner coercive.

\begin{prop}\label{3.1}
  Let $(V,H,j)$ be a pre-evolution triple and
  $1<p<\infty$. Furthermore, let $A(t):V\cap_j H\to (V\cap_j H)^*$, $t\in I$, be a family of
  operators satisfying (\hyperlink{C.1}{C.1}) and
  (\hyperlink{C.2}{C.2}). Then it holds:
  \begin{description}[{(iii)}]
  \item[(i)] If in addition $\{A(t)\}_{t\in I}$ satisfies
    (\hyperlink{C.3}{C.3}), where
    \begin{description}[\textbf{(C.3)}]
    \item[\textbf{(C.3)}]\hypertarget{C.3} For some non-negative
      functions $\alpha,\gamma\in L^{p'}(I)$, $\beta\in L^\infty(I)$
      and a non-decreasing function
      $\mathscr{B}:\mathbb{R}_{\ge 0}\rightarrow \mathbb{R}_{\ge 0}$
      holds
      \begin{align*}
        \left\|A(t)v\right\|_{(V\cap_j H)^*}\leq \mathscr{B}(\|jv\|_H)(\alpha(t)+\beta(t)\|v\|_{V}^{p-1})+\gamma(t)
      \end{align*}
      for almost every $t\in I$ and all $v\in V\cap_j H$.
    \end{description}
    then the induced operator
    $\mathbfcal{A}:D(\bfA) \subseteq
    \mathbfcal{X}\rightarrow\mathbfcal{X}^*$ with
    $D(\bfA) =\mathbfcal{X}\cap_{\boldsymbol{j}} \mathbfcal{Y}$ is
    well-defined. Moreover, $\bfA$ maps bounded sets in
    $\mathbfcal{X}\cap_{\boldsymbol{j}} \mathbfcal{Y}$ into bounded
    sets in $\bfX^*$, i.e., $\bfA$ viewed as an operator from $\mathbfcal{X}\cap_{\boldsymbol{j}}
    \mathbfcal{Y}$ into $\bfX^*$ is bounded. 
  \item[(ii)] If in addition $\{A(t)\}_{t\in I}$
    satisfies (\hyperlink{C.3}{C.3}), (\hyperlink{C.4}{C.4}) and  (\hyperlink{C.5}{C.5}), where
    \begin{description}[\textbf{(C.3)}]
    \item[\textbf{(C.4)}]
      \hypertarget{C.4}{}$A(t):V\cap_j H\rightarrow (V\cap_j H)^*$ is
      pseudo-monotone for almost every $t\in I$.
    \item[\textbf{(C.5)}] \hypertarget{C.5}{} For some constant
      $c_0>0$, non-negative functions
      $c_1,c_2\in L^1(I,\mathbb{R}_{\ge 0})$ and a non-decreasing
      function
      $\mathscr{C}:\mathbb{R}_{\ge 0}\rightarrow \mathbb{R}_{\ge 0}$
      holds
      \begin{align*}
        \langle A(t)v,v\rangle_{V\cap_j H}\ge c_0\|v\|_V^p-c_1(t)\mathscr{C}(\|jv\|_H)-c_2(t)
      \end{align*}
      for almost every $t\in I$ and all $v\in V\cap_j H$.
    \end{description}
    then $\mathbfcal{A}:D(\bfA) \subseteq
    \mathbfcal{X}\rightarrow\mathbfcal{X}^*$ with
    $D(\bfA) =\mathbfcal{X}\cap_{\boldsymbol{j}} \mathbfcal{Y}$ is Bochner 
    pseudo-monotone.
  \item[(iii)] If in addition $\{A(t)\}_{t\in I}$
    satisfies  (\hyperlink{C.3}{C.3}), (\hyperlink{C.4}{C.4}) and (\hyperlink{C.5}{C.5}) with $\mathscr{C}(s)=s^2$, then 
    $\mathbfcal{A}:D(\bfA) \subseteq
    \mathbfcal{X}\rightarrow\mathbfcal{X}^*$ with
    $D(\bfA) =\mathbfcal{X}\cap_{\boldsymbol{j}} \mathbfcal{Y}$ is Bochner coercive.
  \end{description}
\end{prop}

We emphasize that in applications the conditions
(\hyperlink{C.1}{C.1}), (\hyperlink{C.2}{C.2}), (\hyperlink{C.4}{C.4})
and (\hyperlink{C.5}{C.5}) are usually directly deducible from the
corresponding steady problem given trough the operator
$A(t):V\cap_j H\to (V\cap_j H)^*$ for fixed $t\in I$. Only the
verification of (\hyperlink{C.3}{C.3}) sometimes causes some moderate
effort. These circumstances are illustrated in the Examples \ref{5.1}
and \ref{5.2}.

\begin{proof}
  \textbf{ad (i)} \textbf{1. Well-definedness:} Due to Lemma
  \ref{inda} the induced operator
  $\mathbfcal{A}:\mathbfcal{M}(I,V\cap_j H)\to
  \mathbfcal{M}(I,(V\cap_j H)^*)$ is well-defined. Then,
  well-definedness of
  $\mathbfcal{A}: D(\bfA) \subseteq \mathbfcal{X}
  \rightarrow\mathbfcal{X}^*$ with
  $D(\bfA) =\mathbfcal{X}\cap_{\boldsymbol{j}} \mathbfcal{Y}$ follows
  from the estimate
  \begin{align}
    \left(\int_I{\|A(s)(\boldsymbol{x}(s))\|_{(V\cap_j
    H)^*}^{p'}\,ds}\right)^{\frac{1}{p'}}\!\!
    &\leq \left(\int_I{\left(\mathscr{B}(\|j(\boldsymbol{x}(s))\|_H)
      \big(\alpha(t)+\beta(t)\|\boldsymbol{x}(s)\|_{V\cap_j
      H}^{p-1}\big)
      +\gamma(s)\right)^{p'}\!\!ds}\right)^{\frac{1}{p'}} \notag 
    \\
    &\leq\mathscr{B}(\|\boldsymbol{j}\boldsymbol{x}\|_{\mathbfcal{Y}})
    \left(\|\alpha\|_{L^{p'}(I)}+\|\beta\|_{L^\infty(I)}\|\boldsymbol{x}\|_{\mathbfcal{X}}^{p-1}\right)
      +\|\gamma\|_{L^{p'}(I)}\hspace*{-6mm}\label{eq:bbd}
  \end{align}
  for all $\boldsymbol{x}\in\mathbfcal{X} \cap_{\boldsymbol{j}}\mathbfcal{Y}$,
  where we used (\hyperlink{C.3}{C.3}).

  \textbf{2. Boundedness:} The boundedness of
  $\mathbfcal{A}: D(\bfA) \subseteq \mathbfcal{X}
  \rightarrow\mathbfcal{X}^*$ with
  $D(\bfA) =\mathbfcal{X}\cap_{\boldsymbol{j}} \mathbfcal{Y}$ also follows
  from the estimate \eqref{eq:bbd}. 

\textbf{ad (ii)} The presented proof is a generalization of \cite[Lemma 4.2]{BR17}
and uses ideas from \cite{LM87,Hir1,Hir2,Shi97}. Our approach
completely avoids additional technical assumptions on the spaces, as
e.g.~the existence of certain projections, which were present in
previous investigations. We proceed in four steps:

\textbf{1. Collecting information:} Let
$(\boldsymbol{x}_n)_{n\in\setN}\subseteq
\mathbfcal{X}\cap_{\boldsymbol{j}}\mathbfcal{Y}$ be a sequence
satisfying \eqref{eq:3.3}--\eqref{eq:3.5}. Thus,
$(\boldsymbol{x}_n)_{n \in \setN}$
is bounded in $\mathbfcal{X}\cap_{\boldsymbol{j}}\mathbfcal{Y}$, and
due to (i) the sequence $(\bfA\boldsymbol{x}_n)_{n \in \setN}$ is bounded in
$\bfX^*$.  From the reflexivity of $\mathbfcal{X}^*$ we obtain a
subsequence $(\boldsymbol{x}_n)_{n\in\Lambda}
$ 
with
$\Lambda\subseteq\mathbb{N}$ and
$\boldsymbol{\xi}\in\mathbfcal{X}^*$ such that
${\mathbfcal{A}\boldsymbol{x}_n\overset{n\to\infty}{\rightharpoonup}\boldsymbol{\xi}\text{
    in }\mathbfcal{X}^*}$ $(n\in
\Lambda)$ and
$\lim_{\substack{n\rightarrow\infty\\n\in\Lambda}}{\langle
  \mathbfcal{A}\boldsymbol{x}_n,\boldsymbol{x}_n\rangle_{\mathbfcal{X}}}=
\liminf_{n\rightarrow\infty}{\langle\mathbfcal{A}
  \boldsymbol{x}_n,\boldsymbol{x}_n\rangle_{\mathbfcal{X}}}$. Thus, we
have for all $\boldsymbol{y}\in \mathbfcal{X}$
\begin{align}
\lim_{\substack{n\rightarrow\infty\\n\in\Lambda}}{
	\langle \mathbfcal{A}\boldsymbol{x}_n,\boldsymbol{x}_n-\boldsymbol{y}\rangle_{\mathbfcal{X}}}\leq
\liminf_{n\rightarrow\infty}{\langle\mathbfcal{A}
	\boldsymbol{x}_n,\boldsymbol{x}_n-\boldsymbol{y}
	\rangle_{\mathbfcal{X}}}.\label{eq:3.8}
\end{align}
Due to \eqref{eq:3.4} there exists a subset
$E\subseteq I$ such that $I\setminus E$ is a null set and for all $t\in E$
\begin{align}
  (\boldsymbol{j}\boldsymbol{x}_n)(t)
  \overset{n\rightarrow\infty}{\rightharpoonup}
  &(\boldsymbol{j}\boldsymbol{x})(t) \quad \text{  in  }H.\label{eq:3.9} 
\end{align}
In addition, using (\hyperlink{C.3}{C.3}) and
(\hyperlink{C.5}{C.5}) we get
\begin{align*}
  \langle A&(t)(\boldsymbol{x}_n(t)),\boldsymbol{x}_n(t)
             -\boldsymbol{x}(t)\rangle_{V\cap_j H}
  \\
           &\ge c_0\|\boldsymbol{x}_n(t)\|_V^p-
             c_1(t)\mathscr{C}(\|j(\boldsymbol{x}_n(t))\|_H)-c_2(t)
             -\langle
             A(t)(\boldsymbol{x}_n(t)),\boldsymbol{x}(t)\rangle_{V\cap_j H}
  \\
           &\ge 	c_0\|\boldsymbol{x}_n(t)\|_V^p
             -c_1(t)\mathscr{C}(\|j(\boldsymbol{x}_n(t))\|_H)-c_2(t)
  \\
           &\quad -\big (\mathscr{B}(\|j(\boldsymbol{x}_n(t))\|_H)(\alpha(t)
             +\beta(t)\|\boldsymbol{x}_n(t)\|_V^{p-1})+\gamma(t)\big)
             \|\boldsymbol{x}(t)\|_{V\cap_j H}
\end{align*}
for almost every $t\in I$. From
$\|\boldsymbol{j}\boldsymbol{x}_n\|_{\mathbfcal{Y}}\leq K$ for some
constant $K>0$, which follows from \eqref{eq:3.3a}, and the $\varepsilon$-Young inequality with
$k:=k(\varepsilon,p):=(p'\varepsilon)^{1-p}/p$ and
$\varepsilon:=(\mathscr{B}(K)\|\beta\|_{L^\infty(I)})^{-p'}c_0/2$ we
further obtain for all $n\in \Lambda$ and for almost every $t\in I$
\begin{align}
  \langle
  A(t)(\boldsymbol{x}_n(t)),\boldsymbol{x}_n(t)-
  \boldsymbol{x}(t)\rangle_{V\cap_j H}\ge
  \frac{c_0}{2}\|\boldsymbol{x}_n(t)\|_V^p
  -\mu_{\boldsymbol{x}}(t) \tag*{$(\ast)_{n,t}$},
\end{align}
where $\mu_{\boldsymbol{x}}(t):=c_1(t)\mathscr{C}(K)+c_2(t)+k\|\boldsymbol{x}(t)\|_{V\cap_j
	H}^p+(\mathscr{B}(K)\alpha(t)+\gamma(t))\|\boldsymbol{x}(t)\|_{V\cap_j
	H}\in L^1(I)$. Next, we define
\begin{align*}
  \boldsymbol{\mathcal{S}}:=
  \big \{t\in E \fdg &A(t):V\cap_j H\rightarrow (V\cap_j H)^*\text{ is
                       pseudo-monotone},
  \\
                     &\vert\mu_{\boldsymbol{x}}(t)\vert<\infty\text{
                       and }(\ast)_{n,t}\text{ holds for all }n\in\Lambda\big \}. 
\end{align*}
Apparently, $I\setminus\boldsymbol{\mathcal{S}}$ is a null set.

\textbf{2. Intermediate objective:} Our next objective is to verify for all $t\in\mathcal{S}$
\begin{align}
  \liminf_{\substack{n\rightarrow\infty\\n\in\Lambda}}
  {\langle A(t)(\boldsymbol{x}_n(t)),\boldsymbol{x}_n(t)
  -\boldsymbol{x}(t)\rangle_{V\cap_j H}}\ge 0. \tag*{$(\ast\ast )_{t}$}
\end{align}
To this end, let us fix an arbitrary
$t\in\boldsymbol{\mathcal{S}}$ and define
\begin{align*}
\Lambda_t:=\{n\in\Lambda\fdg
\langle A(t)(\boldsymbol{x}_n(t)),\boldsymbol{x}_n(t)
-\boldsymbol{x}(t)\rangle_{V\cap_j H}< 0\}.
\end{align*}
We assume without loss of generality that $\Lambda_t$
is not finite. Otherwise, $(\ast\ast )_{t}$ would already hold true
for this specific $t\in\mathcal{S}$ and nothing would be left to
do. But if $\Lambda_t$ is not finite, then 
\begin{align}
  \limsup_{\substack{n\rightarrow\infty\\n\in\Lambda_t}}
  {\langle A(t)(\boldsymbol{x}_n(t)),\boldsymbol{x}_n(t)-
  \boldsymbol{x}(t)\rangle_{V\cap_j H}}\leq 0.\label{eq:3.12}
\end{align}
From \eqref{eq:3.12} and $(\ast)_{n,t}$ follows for all $n\in\Lambda_t$
\begin{align}
  \frac{c_0}{2}\|\boldsymbol{x}_n(t)\|_V^p\leq\langle A(t)
  (\boldsymbol{x}_n(t)),\boldsymbol{x}_n(t)-
  \boldsymbol{x}(t)\rangle_{V\cap_j H}+
  \vert\mu_{\boldsymbol{x}}(t)\vert<\vert\mu_{\boldsymbol{x}}(t)\vert
  <\infty .\label{eq:3.13}
\end{align}
Thanks to \eqref{eq:3.9} and
\eqref{eq:3.13}, Proposition \ref{6.6} (iv) yields  that
\begin{align*}
  \boldsymbol{x}_n(t)\overset{n\rightarrow\infty}{\rightharpoonup}
  \boldsymbol{x}(t) \quad \text{ in }V\cap_j H\;(n\in \Lambda_t).
\end{align*}
The pseudo-monotonicity of $A(t):V\cap_j H\rightarrow (V\cap_j H)^*$
finally guarantees
\begin{align*}
\liminf_{\substack{n\rightarrow\infty\\n\in\Lambda_t}}
{\langle A(t)(\boldsymbol{x}_n(t)),\boldsymbol{x}_n(t)
	-\boldsymbol{x}(t)\rangle_{V\cap_j H}}\ge 0.
\end{align*}
Due to
$\langle A(t)(\boldsymbol{x}_n(t)), \boldsymbol{x}_n(t)
-\boldsymbol{x}(t)\rangle_{V\cap_j H}\ge 0$ for all
$n\in\Lambda\setminus\Lambda_t$, 
$(\ast\ast)_t$ 
holds for all $t\in\mathcal{S}$.

\textbf{3. Switching to the image space level:} In this passage we
verify the existence of a subsequence
$(\boldsymbol{x}_n)_{n\in\Lambda_0}\subseteq\mathbfcal{X}\cap_{\boldsymbol{j}}\mathbfcal{Y}$
with $\Lambda_0\subseteq\Lambda$ such that for almost every $t\in I$
\begin{equation}
  \begin{aligned}
    \boldsymbol{x}_n(t)\overset{n\rightarrow\infty}{\rightharpoonup}
    \boldsymbol{x}(t) \quad \text{ in }V\cap_j H\;(n\in \Lambda_0),
    \\
    \limsup_{\substack{n\rightarrow\infty\\n\in\Lambda_0}} {\langle
      A(t)(\boldsymbol{x}_n(t)),\boldsymbol{x}_n(t)
      -\boldsymbol{x}(t)\rangle_{V\cap_j H}}\leq 0.
  \end{aligned}
\label{eq:3.14}
\end{equation}
As a consequence, we are in a position to exploit the almost
everywhere pseudo-monotonicity of the operator family.  Thanks to
$\langle A(t)(\boldsymbol{x}_n(t)),\boldsymbol{x}_n(t)
-\boldsymbol{x}(t)\rangle_{V\cap_j H}\ge -\mu_{\boldsymbol{x}}(t)$ for
all $t\in\mathcal{S}$ and $n\in\Lambda$, Fatou's lemma
(cf.~\cite[Theorem 1.18]{Rou05}) is applicable. It yields, also using
$(\ast\ast)_t$ and \eqref{eq:3.5}
\begin{align}
\begin{split}
0&\leq
\int_I{\liminf_{\substack{n\rightarrow\infty\\n\in\Lambda}}
	{\langle A(s)(\boldsymbol{x}_n(s)),\boldsymbol{x}_n(s)-
		\boldsymbol{x}(s)\rangle_{V\cap_j H}}\,ds}
\\
&\leq
\liminf_{\substack{n\rightarrow\infty\\n\in\Lambda}}{\int_I{\langle
		A(s)(\boldsymbol{x}_n(s)),\boldsymbol{x}_n(s)-\boldsymbol{x}(s)\rangle_{V\cap_j
			H}\,ds}}\leq  \limsup_{n\rightarrow\infty}
{\langle\mathbfcal{A}\boldsymbol{x}_n,\boldsymbol{x}_n
	-\boldsymbol{x}\rangle_{\mathbfcal{X}}}\leq  0.
\end{split}\label{eq:3.15}
\end{align}
Let us define  $h_n(t):=\langle
A(t)(\boldsymbol{x}_n(t)),\boldsymbol{x}_n(t)-\boldsymbol{x}(t)\rangle_{V\cap_j
	H}$. Then $(\ast\ast)_t$ and 
\eqref{eq:3.15} read:
\begin{align}
\liminf_{\substack{n\rightarrow\infty\\n\in\Lambda}}{h_n(t)}&\ge 0\text{ for all }t\in\mathcal{S}.\label{eq:3.16}\\
\lim_{\substack{n\rightarrow\infty\\n\in\Lambda}}{\int_I{h_n(s)\,ds}}&=0.\label{eq:3.17}
\end{align}
As $s\mapsto s^-:=\min\{0,s\}$ is continuous and non-decreasing we
deduce from \eqref{eq:3.16} that
\begin{align*}
0\ge\limsup_{\substack{n\rightarrow\infty\\n\in\Lambda}}
{h_n(t)^-}\ge\liminf_{\substack{n\rightarrow\infty\\n\in\Lambda}}
{h_n(t)^-}\ge \min\left\{0,
\liminf_{\substack{n\rightarrow\infty\\n\in\Lambda}}{h_n(t)}\right\}=0,
\end{align*}
i.e., $h_n(t)^-\overset{n\rightarrow\infty}{\rightarrow} 0$
$(n\in \Lambda)$ for all $t\in \mathcal{S}$. Since 
$0\ge h_n(t)^-\ge -\mu_{\boldsymbol{x}}(t)$ for all $t\in\mathcal{S}$ and
$n\in\Lambda$, Vitali's theorem yields
$h_n^-\overset{n\rightarrow\infty}{\rightarrow}0$ in $L^1(I)$.  From
the latter, $\vert h_n\vert=h_n-2h_n^-$ and \eqref{eq:3.17}, we
conclude that $h_n\overset{n\rightarrow\infty}{\rightarrow}0$ in
$L^1(I)$.  Thus, there exists a subsequence
$(\boldsymbol{x}_n)_{n\in\Lambda_0}$
with $\Lambda_0\subseteq\Lambda$ and a subset $F\subseteq I$ such
that $ I\setminus F$ is a null set and for all $t\in F$
\begin{align}
\lim_{\substack{n\rightarrow\infty\\n\in\Lambda_0}}{\langle
  A(t)(\boldsymbol{x}_n(t)),\boldsymbol{x}_n(t)-\boldsymbol{x}(t)\rangle_{V\cap_j
  H}}= 0. \label{eq:3.18}
\end{align}
Consequently, we have for all $t\in \boldsymbol{\mathcal{S}}\cap F$ 
\begin{align*}
\limsup_{\substack{n\rightarrow\infty\\n\in\Lambda_0}}
{\frac{c_0}{2}\|\boldsymbol{x}_n(t)\|_V^p}\leq
\limsup_{\substack{n\rightarrow\infty\\n\in\Lambda_0}}
{\langle A(t)(\boldsymbol{x}_n(t)),\boldsymbol{x}_n(t)
	-\boldsymbol{x}(t)\rangle_{V\cap_j H}+
	\vert\mu_{\boldsymbol{x}}(t)\vert}
=\vert\mu_{\boldsymbol{x}}(t)\vert<\infty.
\end{align*}
Thus, due to \eqref{eq:3.9}, Proposition \ref{6.6} (iv) yields 
\begin{align}
\boldsymbol{x}_n(t)\overset{n\rightarrow\infty}
{\rightharpoonup}\boldsymbol{x}(t) \quad \text{ in }
V\cap_j H\;(n\in \Lambda_0)\label{eq:3.19}
\end{align} 
for all $t\in \boldsymbol{\mathcal{S}}\cap F$. The relations 
\eqref{eq:3.18} and \eqref{eq:3.19} imply \eqref{eq:3.14}.

\textbf{4. Switching to the Bochner-Lebesgue level:} In view of the
pseudo-monotonicity of the operators $A(t):V\cap_j H\rightarrow (V\cap_j H)^*$ for
all $t\in\boldsymbol{\mathcal{S}}\cap F$ we deduce from
\eqref{eq:3.14} that 
\begin{align*}
\langle A(t)(\boldsymbol{x}(t)),\boldsymbol{x}(t)-
\boldsymbol{y}(t)\rangle_{V\cap_j H}\leq
\liminf_{\substack{n\rightarrow\infty\\n\in\Lambda_0}}
{\langle A(t)(\boldsymbol{x}_n(t)),\boldsymbol{x}_n(t)
	-\boldsymbol{y}(t)\rangle_{V\cap_j H}}
\end{align*}
almost every $t\in I$ and all $\boldsymbol{y}\in \mathbfcal{X}$. As
in step {\bf 1} we verify that
there exists $\mu_{\boldsymbol{y}}\in L^1(I)$ such that
\begin{align*}
\langle A(t)(\boldsymbol{x}_n(t)),\boldsymbol{x}_n(t)
-\boldsymbol{y}(t)\rangle_{V\cap_j H}\ge
\frac{c_0}{2}\|\boldsymbol{x}_n(t)\|_V^p-\mu_{\boldsymbol{y}}(t)
\end{align*}
for almost every $t\in I$ and all $n\in\Lambda_0$. Thus,
we can apply Fatou's lemma once more, exploit \eqref{eq:3.8} and deduce further that
\begin{align*}
\langle \mathbfcal{A}\boldsymbol{x},\boldsymbol{x}
-\boldsymbol{y}\rangle_{\mathbfcal{X}}
&\leq
\int_I{\liminf_{\substack{n\rightarrow\infty\\n\in\Lambda_0}}
	{\langle A(s)(\boldsymbol{x}_n(s)),\boldsymbol{x}_n(s)
		-\boldsymbol{y}(s)\rangle_{V\cap_j H}}\,ds}
\\
&\leq \liminf_{\substack{n\rightarrow\infty\\n\in\Lambda_0}}
{\int_I{\langle A(s)(\boldsymbol{x}_n(s)),\boldsymbol{x}_n(s)
		-\boldsymbol{y}(s)\rangle_{V\cap_j H}\,ds}}
\\
&=\lim_{\substack{n\rightarrow\infty\\n\in\Lambda}}
{\langle \mathbfcal{A}\boldsymbol{x}_n,\boldsymbol{x}_n
	-\boldsymbol{y}\rangle_{\mathbfcal{X}}}
\\
&\leq \liminf_{n\rightarrow\infty}{\langle \mathbfcal{A}
	\boldsymbol{x}_n,\boldsymbol{x}_n-\boldsymbol{y}\rangle_{\mathbfcal{X}}}
\end{align*}
for all $\boldsymbol{y}\in \mathbfcal{X}$,
i.e.,  $\mathbfcal{A}: D(\bfA) \subseteq \mathbfcal{X}
  \rightarrow\mathbfcal{X}^*$ with
  $D(\bfA) =\mathbfcal{X}\cap_{\boldsymbol{j}} \mathbfcal{Y}$ 
is Bochner pseudo-monotone.

\textbf{ad (iii)} As in the proof of Lemma \ref{3.21} it suffices to show that  $\mathbfcal{A}: D(\bfA) \subseteq \mathbfcal{X}
  \rightarrow\mathbfcal{X}^*$ with
  $D(\bfA) =\mathbfcal{X}\cap_{\boldsymbol{j}} \mathbfcal{Y}$
is Bochner coercive with respect to the origin
$ \boldsymbol{0}\in\mathbfcal{X}^*$ and $\boldsymbol{x}_0\in H$. To show that
 $\mathbfcal{A}: D(\bfA) \subseteq \mathbfcal{X}
  \rightarrow\mathbfcal{X}^*$ with
  $D(\bfA) =\mathbfcal{X}\cap_{\boldsymbol{j}} \mathbfcal{Y}$ is Bochner coercive with respect to
$\boldsymbol 0$ and $\boldsymbol{x}_0$, we assume
that
$\boldsymbol{x}\in\mathbfcal{X}\cap_{\boldsymbol{j}}\mathbfcal{Y}$
satisfies for almost every $t\in I$
\begin{align}
\frac{1}{2}\|(\boldsymbol{j}\boldsymbol{x})(t)\|_H^2+\int_0^t{\left\langle
	A(s)(\boldsymbol{x}(s)),\boldsymbol{x}(s)\right\rangle_{V\cap_j
		H}\,ds}\leq \frac{1}{2}\|\boldsymbol{x}_0\|_H^2.\label{eq:3.24} 
\end{align}
Using (\hyperlink{C.5}{C.5}) with $\mathscr{C}(s)=s^2$ in
\eqref{eq:3.24} we get   for almost every $t\in I$
\begin{align}
\frac{1}{2}\|(\boldsymbol{j}\boldsymbol{x})(t)\|_H^2+c_0\int_{0}^{t}{\|\boldsymbol{x}(s)\|_V^p\,ds}\leq
\frac{1}{2}\|\boldsymbol{x}_0\|_H^2+ \|c_2\|_{L^1(I)}+\int_0^t{\vert
	c_1(s)\vert\|(\boldsymbol{j}\boldsymbol{x})(s)\|_H^2\,ds}. \label{eq:3.25} 
\end{align}
Gronwall's inequality (cf. \cite[Lemma II.4.10]{BF13}) applied on
\eqref{eq:3.25} yields
\begin{align}
\|\boldsymbol{j}\boldsymbol{x}\|_{\mathbfcal{Y}}^2
\leq\|\boldsymbol{x}_0\|_H^2+2\|c_2\|_{L^1(I)}\text{exp}(2\|c_1\|_{L^1(I)})
=:K_0.\label{eq:3.26} 
\end{align}
From \eqref{eq:3.25}  and \eqref{eq:3.26} we further deduce that
\begin{align}
c_0\int_0^t{\|\boldsymbol{x}(s)\|_V^p\,ds}\leq
\frac{1}{2}\|\boldsymbol{x}_0\|_H^2+\|c_2\|_{L^1(I)}+K_0\|c_1\|_{L^1(I)}=:K_1\label{eq:3.27}
\end{align}
for all $t\in \overline{I}$. \eqref{eq:3.26} together with
\eqref{eq:3.27} reads
$\|\boldsymbol{x}\|_{L^p(I,V)\cap_{\boldsymbol{j}}\mathbfcal{Y}}\leq
(K_1/c_0)^{\frac{1}{p}}+K_0^{\frac{1}{2}}$. Due to the norm
equivalence
$\|\cdot\|_{L^p(I,V)\cap_{\boldsymbol{j}}\mathbfcal{Y}}\sim
\|\cdot\|_{\mathbfcal{X}\cap_{\boldsymbol{j}}\mathbfcal{Y}}$
(cf. Proposition \ref{2.6}) we conclude the Bochner coercivity with
respect to $\boldsymbol{0}\in\mathbfcal{X}^*$ and $\boldsymbol{x}_0\in H$ of
 $\mathbfcal{A}: D(\bfA) \subseteq \mathbfcal{X}
  \rightarrow\mathbfcal{X}^*$ with
  $D(\bfA) =\mathbfcal{X}\cap_{\boldsymbol{j}} \mathbfcal{Y}$.
\hfill$\square$
\end{proof}

\section{Existence theorem}
\label{sec:4}
\begin{thm}[Main theorem]\label{4.1}
  Let $(V,H,j)$ be a pre-evolution triple, $1<p<\infty$ and
  $A(t):V\cap_j H\to (V\cap H)^*$, $t\in I$, a family of operators such that
  (\hyperlink{C.1}{C.1})--(\hyperlink{C.3}{C.3}) are fulfilled and
  that the induced operator  $\mathbfcal{A}: D(\bfA) \subseteq \mathbfcal{X}
  \rightarrow\mathbfcal{X}^*$ with
  $D(\bfA) =\mathbfcal{X}\cap_{\boldsymbol{j}} \mathbfcal{Y}$
  is Bochner pseudo-monotone and Bochner coercive with respect to
  $\boldsymbol{f}\in \mathbfcal{X}^*$ and $\boldsymbol{y}_0\in H$. Then exists a solution
  $\boldsymbol{y}\in \mathbfcal{W}$ of the evolution
  equation \eqref{eq:1}, i.e.,
  \begin{align*}
    \begin{aligned}
      \frac{d_e\boldsymbol{y}}{dt}+\mathbfcal{A}\boldsymbol{y}&=\boldsymbol{f}\quad&&\text{ in }\mathbfcal{X}^*,\\
      (\boldsymbol{j}\boldsymbol{y})(0)&=\boldsymbol{y}_0&&\text{ in }H.
    \end{aligned}
  \end{align*}
\end{thm}

From Lemma \ref{3.1} we immediately obtain the following more applicable version of Theorem~\ref{4.1}.

\begin{cor}\label{4.2}
  Let $(V,H,j)$ be an pre-evolution triple, $1<p<\infty$ and
  $A(t):V\cap_j H\to (V\cap H)^*$, $t\in I$, a family of operators such that
  (\hyperlink{C.1}{C.1})--(\hyperlink{C.5}{C.5}) are fulfilled with
  $\mathscr{C}(s)=s^2$ in (\hyperlink{C.5}{C.5}).  Then for arbitrary
  $\boldsymbol{y}_0\in H$ and $\boldsymbol{f}\in\mathbfcal{X}^*$ there
  exists a solution $\boldsymbol{y}\in \mathbfcal{W}$ of the evolution
  equation \eqref{eq:1}.
\end{cor}

\begin{rmk}\label{4.3}
  If $(V,H,j)$ is an evolution triple, the assertions of Theorem
  \ref{4.1} and Corollary \ref{4.2} remain true as $(V,H,j)$ is a
  pre-evolution triple as well. In addition, one can replace
  $V\cap_j H$ by $V$ in Theorem \ref{4.1} and Corollary \ref{4.2} as
  $V=V\cap_j H$ with norm equivalence.
\end{rmk}
\paragraph{\textbf{Proof} (of Theorem \ref{4.1})}

\paragraph{\textbf{0. Reduction of assumptions:}} As in the proofs of Lemma \ref{3.21} and Proposition \ref{3.1} (i) it suffices anew to treat the special case $\boldsymbol{f}=\boldsymbol{0}$
in $\mathbfcal{X}^*$. 

\paragraph{\textbf{1. Galerkin approximation:}}
Based on the separability of $V\cap_j H$
(cf.~Proposition~\ref{6.6}~(i)) there exists a sequence
$(v_i)_{i\in \mathbb{N}}\subseteq V\cap_j H$ which is dense in
$V\cap_j H$. Due to the density of $R(j)$ in $H$ and the Gram-Schmidt
process we can additionally assume that
$(jv_i)_{i\in \mathbb{N}}\subseteq H$ is dense and orthonormal
in $H$. We set $V_n:=\text{span}\{v_1,...,v_n\}$ equipped with
$\|\cdot\|_V$ and $H_n:=j(V_n)$ equipped with $(\cdot,\cdot)_H$.
Denote by $j_n:V_n\to H_n$ the restriction of $j$ to $V_n$ and by
$R_n:H_n\to H_n^*$ the corresponding Riesz isomorphism with respect to
$(\cdot,\cdot)_H$. As $j_n$ is an isomorphism, the triple
$(V_n,H_n,j_n)$ is an evolution triple with canonical embedding
$e_n:=j_n^*R_nj_n:V_n\to V_n^*$.  Moreover, we set
\begin{align*}
  \mathbfcal{X}_n:=L^p(I,V_n),\qquad
  \mathbfcal{W}_n:=W_{e_n}^{1,p,p'}(I,V_n,V_n^*),\qquad
  \mathbfcal{Y}_n:=C^0(\overline{I},H_n). 
\end{align*}
Then Proposition \ref{2.11} provides the embedding  $\boldsymbol{j}_n:\mathbfcal{W}_n\to\mathbfcal{Y}_n$ and the generalized 
integration by parts formula with respect to $\mathbfcal{W}_n$.

We are seeking approximative solutions
$\boldsymbol{y}_n\in \mathbfcal{W}_n$ which solve the Galerkin system
\begin{align}\label{eq:4.4}
  \begin{aligned}
    \frac{d_{e_n}\boldsymbol{y}_n}{dt}+(\text{id}_{\mathbfcal{X}_n})^*
    \mathbfcal{A}\boldsymbol{y}_n&=\boldsymbol{0}&&\text{ in
    }\mathbfcal{X}_n^*,
    \\
    (\boldsymbol{j}_n\boldsymbol{y}_n)(0)&=\boldsymbol{y}_0^n &&\text{
      in }H_n,
  \end{aligned}
\end{align}
where $\boldsymbol{y}_0^n:=\sum_{i=1}^{n}{(\boldsymbol{y}_0,jv_i)_Hjv_i}$.
\paragraph{\textbf{2. Existence of Galerkin solutions:}}

It is straightforward to check that $\boldsymbol{y}_n \in \mathbfcal{W}_n$ iff 
\begin{align}
    \boldsymbol{y}_n=\sum_{i=1}^{n}{\alpha_i^nv_i}\text{ with }
  \alpha_i^n\in W^{1,p'}(I) \qquad \text{ and } \qquad
  \frac{d_{e_n}\boldsymbol{y}_n}{dt}=\sum_{i=1}^{n}
  {\frac{d\alpha_i^n}{dt}e_nv_i}\text{ in }\mathbfcal{X}_n^*. \label{eq:4.5}
\end{align}
Thus, defining $\boldsymbol{f}^n:I\times\mathbb{R}^n\to \mathbb{R}^n$
by
$\boldsymbol{f}^n(t,\boldsymbol \alpha):=(\langle
A(t)(\sum_{k=1}^n{\alpha_kv_k}),v_i\rangle_{V\cap_j H})_{i=1,...,n}$
for almost every $t\in I$ and all
$\boldsymbol{\alpha}=(\alpha_i)_{i=1,...,n}\in \mathbb{R}^n$, one
sees, that \eqref{eq:4.4} can be re-written as a system of ordinary
differential equations
\begin{align}
  \begin{alignedat}{2}
    \frac{d\boldsymbol{\alpha}^n}{dt}(s)&=
    \boldsymbol{f}^n(s,\boldsymbol{\alpha}^n(s))&&\text{ in }\mathbb{R}^n\text{ for a.e. }s\in I,\\
    \boldsymbol{\alpha}^n(0)&=((\boldsymbol{y}_0,jv_i))_{i=1,...,n}\quad
    &&\text{
      in }\mathbb{R}^n.
  \end{alignedat}\label{eq:4.6}
\end{align} 
From the assumptions (\hyperlink{C.1}{C.1}) and (\hyperlink{C.2}{C.2})
we deduce that the system \eqref{eq:4.6} satisfies the standard
Carath\'eodory conditions and by assumption (\hyperlink{C.3}{C.3})
additionally a local majorant condition required in Carath\'eodory's
existence theorem (cf. \cite[Theorem 5.2]{Hal80}). The latter provides
a maximal time horizon $T_n\in (0, T]$ and an absolutely continuous
solution
$\boldsymbol{\alpha}^n:\left[0,T_n\right)\rightarrow \mathbb{R}^n$ of
\eqref{eq:4.6} restricted to $\left[0,T_n\right)$. From
(\hyperlink{C.3}{C.3}) and
$\boldsymbol{\alpha}^n\in C^0(\left[0,t\right],\mathbb{R}^n)$ for all
$0<t<T_n$ we infer that
$\frac{d\boldsymbol{\alpha}^n}{dt}=\boldsymbol{f}^n(\cdot,\boldsymbol{\alpha}^n)\in
L^{p'}(\left(0,t\right),\mathbb{R}^n)$ for all $0<t<T_n$. We set
$\boldsymbol{y}_n:=\sum_{i=1}^{n}{\alpha_i^nv_i}$. Then
$\boldsymbol{y}_n\in W^{1,p,p'}_{e_n}(\left(0,t\right),V_n,V_n^*)$ for
all $0<t<T_n$ (cf.~\eqref{eq:4.5}). Suppose $T_n<T$. We integrate the
inner product of \eqref{eq:4.6} and
$\boldsymbol{\alpha}^n(s)\in \mathbb{R}^n$ with respect to
$s\in\left[0,t\right]$, where $0<t\leq T_n$, apply the generalized
integration by parts formula with respect to
$W^{1,p,p'}_{e_n}(\left(0,t\right),V_n,V_n^*)$
(cf.~Proposition~\ref{2.11}), and use $\boldsymbol j_n =\boldsymbol j$
on $W^{1,p,p'}_{e_n}(\left(0,t\right),V_n,V_n^*)$, to obtain
\begin{align}
\frac{1}{2}\|(\boldsymbol{j}\boldsymbol{y}_n)(t)\|_H^2+\int_{0}^{t}{\langle A(s)(\boldsymbol{y}_n(s)),\boldsymbol{y}_n(s)\rangle_{V\cap_j H}\,ds}\leq\frac{1}{2}\|\boldsymbol{y}_0^n\|_H^2\leq \frac{1}{2}\|\boldsymbol{y}_0\|_H^2\label{eq:4.7}
\end{align}
for all $t\in \left[0,T_n\right)$.  By
$\overline{\boldsymbol{y}}_n:\overline{I}\rightarrow V_n$ we denote
the extension of $\boldsymbol{y}_n:\left[0,T_n\right)\rightarrow V_n$
by zero outside $\left[0,T_n\right)$. Thus, our extension satisfies
\begin{align}
\frac{1}{2}\|(\boldsymbol{j}\overline{\boldsymbol{y}}_n)(t)\|_H^2+\int_{0}^{t}{\langle A(s)(\overline{\boldsymbol{y}}_n(s)),\overline{\boldsymbol{y}}_n(s)\rangle_{V\cap_j H}\,ds}\leq \frac{1}{2}\|\boldsymbol{y}_0\|_H^2\label{eq:4.8}
\end{align}
for all $t\in \overline{I}$. From \eqref{eq:4.8} and the Bochner
coercivity with respect to $\boldsymbol{0}\in\mathbfcal{X}^*$ and $\boldsymbol{y}_0\in H$ of
$\mathbfcal{A}: D(\bfA) \subseteq \mathbfcal{X}
  \rightarrow\mathbfcal{X}^*$ with
  $D(\bfA) =\mathbfcal{X}\cap_{\boldsymbol{j}} \mathbfcal{Y}$
we obtain an $n$-independent constant $M>0$ such that
\begin{align*}
\|\boldsymbol{y}_n\|_{L^p(\left(0,T_n\right),V\cap_j
  H)\cap_{\boldsymbol{j}}L^\infty(\left(0,T_n\right),H)}=
  \|\overline{\boldsymbol{y}}_n\|_{\mathbfcal{X}\cap_{\boldsymbol{j}}\mathbfcal{Y}}
  \leq M. 
\end{align*}
In consequence, $\boldsymbol{\alpha}^n\in
L^\infty(\left(0,T_n\right),\mathbb{R}^n)$ and therefore
$\frac{d\boldsymbol{\alpha}^n}{dt}=\boldsymbol{f}^n(\cdot,\boldsymbol{\alpha}^n)\in
L^{p'}(\left(0,T_n\right),\mathbb{R}^n)$ due to (\hyperlink{C.3}{C.3}). The fundamental theorem of
calculus now yields $\boldsymbol{\alpha}^n\in
C^0(\left[0,T_n\right],\mathbb{R}^n)$. Hence, we can apply
Caratheodory's theorem once more with initial value
$\boldsymbol{\alpha}^n(T_n)\in \mathbb{R}^n$, to obtain an extension of
$\boldsymbol{\alpha}^n$ to a solution of \eqref{eq:4.6} on
$\left[0,T_n+\varepsilon\right]$, with $\varepsilon>0$. This
contradicts the maximality of $T_n>0$ and we conclude $T_n=T$. In
particular, the estimates
\begin{align}
  \|\boldsymbol{y}_n\|_{\mathbfcal{X}\cap_{\boldsymbol{j}}\mathbfcal{Y}}\leq
  M\qquad \text{ and }\qquad \|\mathbfcal{A}\boldsymbol{y}_n\|_{\mathbfcal{X}^*}\leq
  M'\label{eq:4.10} 
\end{align}
hold true, where we used the boundedness of
$\mathbfcal{A}:\mathbfcal{X}\cap_{\boldsymbol{j}}\mathbfcal{Y}\rightarrow\mathbfcal{X}^*$
according to Lemma \ref{3.1}~(i) for the second estimate.

\paragraph{\textbf{3. Passage to the limit:}}

\paragraph{\textbf{3.1 Convergence of the Galerkin solutions:}}
From the a-priori estimates \eqref{eq:4.10} we obtain a not relabelled
subsequence $(\boldsymbol{y}_n)_{n\in \mathbb{N}}\subseteq
\mathbfcal{X}\cap_{\boldsymbol{j}}\mathbfcal{Y}$ as well as elements
$\boldsymbol{y}\in\mathbfcal{X}\cap_{\boldsymbol{j}}\mathbfcal{Y}$ and
$\boldsymbol{\xi}\in\mathbfcal{X}^*$ such that
\begin{align}
  \begin{alignedat}{2}
    \boldsymbol{y}_n&\overset{n\rightarrow\infty}{\rightharpoonup}\boldsymbol{y}&\quad
    &\text{ in }\mathbfcal{X},\\
    \boldsymbol{j}\boldsymbol{y}_n&\;\;\overset{\ast}{\rightharpoondown}\;\;\boldsymbol{j}\boldsymbol{y}&&\text{
      in }\mathbfcal{Y},\qquad (n\rightarrow\infty),\\
    \mathbfcal{A}\boldsymbol{y}_n&\overset{n\rightarrow\infty}{\rightharpoonup}\boldsymbol{\xi}&&\text{
      in }\mathbfcal{X}^*.
\end{alignedat}\label{eq:4.11}
\end{align}

\paragraph{\textbf{3.2 Regularity and trace of the weak limit:}}
\hypertarget{3.2}{}Let $v\in V_k$, $k\in \mathbb{N}$, and
$\varphi\in C^\infty(\overline{I})$ with $\varphi(T)=0$. Testing
\eqref{eq:4.4} for $n\ge k$ by
$v\varphi\in \mathbfcal{X}_k\subseteq \mathbfcal{X}_n$ and a
subsequent application of the generalized integration by parts formula
with respect to $\mathbfcal{W}_n$ (cf. Proposition \ref{2.11}) yields
\begin{align*}
  \int_I{\langle A(s)(\boldsymbol{y}_n(s)),v\rangle_{V\cap_j
  H}\varphi(s)\,ds} =\int_I{((\boldsymbol{jy}_n)(s),jv)_H\varphi\prime(s)\,ds}
  +(\boldsymbol y_0^n,jv)_H\varphi(0).
\end{align*}
By passing with $n\ge k$ to infinity, using \eqref{eq:4.11} and
$\boldsymbol y_0^n\overset{n\rightarrow\infty} {\rightarrow}\boldsymbol{y}_0$ in
$H$, we obtain
\begin{align}
\int_I{\langle\boldsymbol{\xi}(s),v\rangle_{V\cap_j
  H}\varphi(s)\,ds}=\int_I{((\boldsymbol{jy})(s),jv)_H\varphi\prime(s)\,ds}
  +(\boldsymbol{y}_0,jv)_H\varphi(0)\label{eq:4.12} 
\end{align}
for all $v\in \bigcup_{k\in\mathbb{N}}{V_k}$ and
$\varphi\in C^\infty(\overline{I})$ with $\varphi(T)=0$. Choosing  
$\varphi\in C_0^\infty(I)$ in \eqref{eq:4.12}, we have due to Definition \ref{2.10} and Proposition \ref{2.11}
\begin{align}\label{eq:W}
  \boldsymbol{y}\in \mathbfcal{W} \textrm{ with }
  \frac{d_e\boldsymbol{y}}{dt}=-\boldsymbol{\xi} \quad\textrm { in }
  \mathbfcal{X}^*\qquad \textrm{ and } \qquad \boldsymbol{j}\boldsymbol{y}\in C^0(\overline{I},H).
\end{align}
Thus, we are allowed to apply the generalized integration by parts formula with respect to $\mathbfcal{W}$ in
\eqref{eq:4.12} in the case $\varphi\in C^\infty(\overline{I})$ with
$\varphi(T)=0$ and $\varphi(0)=1$, which yields for all $v\in
\bigcup_{k\in\mathbb{N}}{V_k}$ 
\begin{align}
  ((\boldsymbol{jy})(0)-\boldsymbol{y}_0,jv)_H=0. \label{eq:4.13}
\end{align}
As $R(j)$ is dense in $H$ we deduce from \eqref{eq:4.13} that
\begin{align}
  (\boldsymbol{jy})(0)=\boldsymbol{y}_0\quad \text{ in }H.\label{eq:4.14}
\end{align}
\paragraph{\textbf{3.3 Pointwise weak convergence in $H$:}}

Now we show that
$ (\boldsymbol{j}\boldsymbol{y}_n)(t)
\overset{n\rightarrow\infty}{\rightharpoonup} (\boldsymbol j
\boldsymbol{y})(t)$ in $H$ for almost every $t \in I$, which is the
crucial new condition of Bochner pseudo-monotonicity compared to
standard pseudo-monotonicity, apart from the boundedness in
$\mathbfcal{Y}$.  To this end, let us fix an arbitrary
$t\in\left(0,T\right]$. From the a-priori estimate
$\|(\boldsymbol{j}\boldsymbol{y}_n)(t)\|_{H}\leq M$ for all
$t\in \overline{I}$ and $n\in\mathbb{N}$ (cf.~\eqref{eq:4.10}) we
obtain the existence of a subsequence
$((\boldsymbol{j}\boldsymbol{y}_n)(t))_{n\in\Lambda_t}\subseteq
H$ with $\Lambda_t\subseteq\mathbb{N}$, initially
depending on this fixed $t$, and an element
$\boldsymbol{y}_{\Lambda_t}\in H$ such that
\begin{align}
  (\boldsymbol{j}\boldsymbol{y}_n)(t)
  &\overset{n\rightarrow\infty}{\rightharpoonup}
    \boldsymbol{y}_{\Lambda_t}\quad \text{ in }H\text{ }(n\in
    \Lambda_t).\label{eq:4.15}  
\end{align}
For $v\in V_k$, $k\in\Lambda_t$, and
$\varphi\in C^\infty(\overline{I})$ with $\varphi(0)=0$ and
$\varphi(t)=1$, we test \eqref{eq:4.4} for $n\ge k$
($n\in \Lambda_t$) by
$v\varphi\chi_{\left[0,t\right]}\in \mathbfcal{X}_k\subseteq
\mathbfcal{X}_n$, use the
generalized integration by parts formula in $\mathbfcal{W}_n$, and
\eqref{eq:iden}, to obtain for all
$n\ge k$ with $n\in \Lambda_t$ 
\begin{align*}
  \int_0^t{\langle A(s)(\boldsymbol{y}_n(s)),v\rangle_{V\cap_j
  H}\varphi(s)\,ds}=\int_0^t{((\boldsymbol{j}\boldsymbol{y}_n)(s),
  jv)_H\varphi\prime(s)\,ds}-((\boldsymbol{j}\boldsymbol{y}_n)(t),jv)_H .
\end{align*}
By passing for $n\ge k$ with $n\in \Lambda_t$ to
infinity, using \eqref{eq:4.11} and \eqref{eq:4.15}, we obtain 
\begin{align*}
  \int_0^t{\langle \boldsymbol{\xi}(s),v\rangle_{V\cap_j
  H}\varphi(s)\,ds}=\int_0^t{((\boldsymbol{j}\boldsymbol{y})(s),
  jv)_H\varphi\prime(s)\,ds}-(\boldsymbol{y}_{\Lambda_t},jv)_H 
\end{align*}
for all $v\in \bigcup_{k\in\Lambda_t}{V_k}$. From
\eqref{eq:W} and the generalized integration by parts formula in $\mathbfcal{W}$ we also
obtain 
\begin{align}
  ((\boldsymbol{j}\boldsymbol{y})(t)-
  \boldsymbol{y}_{\Lambda_t},jv)_H=0\label{eq:4.16} 
\end{align}
for all $v\in\bigcup_{k\in\Lambda_t}{V_k}$. Thanks to
$V_k\subseteq V_{k+1}$ for all $k\in\mathbb{N}$ we get
$\bigcup_{k\in\Lambda_t}{V_k}=\bigcup_{k\in\mathbb{N}}{V_k}$. Thus,
$j(\bigcup_{k\in\Lambda_t}{V_k})$ is dense in $H$ and
\eqref{eq:4.16} yields that
$(\boldsymbol{j}\boldsymbol{y})(t)=\boldsymbol{y}_{\Lambda_t}$
in $H$. Consequently, we deduce from \eqref{eq:4.15} that 
\begin{align}
  (\boldsymbol{j}\boldsymbol{y}_n)(t)\overset{n\rightarrow\infty}
  {\rightharpoonup}(\boldsymbol{j}\boldsymbol{y})(t)\quad \text{
  in }H\text{ }(n\in\Lambda_t).\label{eq:4.17} 
\end{align}
As this argumentation stays valid for each weakly convergent subsequence of
$((\boldsymbol{j}\boldsymbol{y}_n)(t) )_{n\in\mathbb{N}}\subseteq H$,
$(\boldsymbol{j}\boldsymbol{y})(t)\in H$ is weak accumulation point of
each weakly converging subsequence of
$((\boldsymbol{j}\boldsymbol{y}_n)(t) )_{n\in\mathbb{N}}\subseteq
H$. The standard convergence principle (cf.~\cite[Kap.~I, Lemma 5.4]{GGZ74}) yields
$\Lambda_t=\mathbb{N}$ in \eqref{eq:4.17}.
\paragraph{\textbf{3.4 Identification of $\mathbfcal{A}\boldsymbol{y}$ and $\boldsymbol{\xi}$}:}
Due to \eqref{eq:4.7} in the case $t=T$ we have for all $n\in \mathbb{N}$
\begin{align*}
\langle
  \mathbfcal{A}\boldsymbol{y}_n,\boldsymbol{y}_n\rangle_\mathbfcal{X}\leq
  -\frac{1}{2}\|(\boldsymbol{j}\boldsymbol{y}_n)(T)\|_H^2+
  \frac{1}{2}\|\boldsymbol{y}_0\|_H^2 .
\end{align*}
This inequality together with 
$\eqref{eq:4.11}_3$, \eqref{eq:4.14}, \eqref{eq:4.17} with
$\Lambda_t=\mathbb{N}$ in the case $t=T$, the weak lower
semi-continuity of $\|\cdot\|_H$, the generalized integration by parts
formula in $\mathbfcal{W}$ and \eqref{eq:W}
yields
\begin{align}\begin{split}
    \limsup_{n\rightarrow\infty}{ \langle
      \mathbfcal{A}\boldsymbol{y}_n,\boldsymbol{y}_n-\boldsymbol{y}\rangle_\mathbfcal{X}}
    &\leq
    -\frac{1}{2}\|(\boldsymbol{j}\boldsymbol{y})(T)\|_H^2+\frac{1}{2}\|(\boldsymbol{j}\boldsymbol{y})(0)\|_H^2-\langle
    \boldsymbol{\xi},\boldsymbol{y}\rangle_\mathbfcal{X}\\&=-\left\langle
      \frac{d_e\boldsymbol{y}}{dt},\boldsymbol{y}\right\rangle_\mathbfcal{X}-\langle
    \boldsymbol{\xi},\boldsymbol{y}\rangle_\mathbfcal{X}=0.
\end{split}\label{eq:4.18}
\end{align}
As a result of \eqref{eq:4.11}, \eqref{eq:4.17} with
$\Lambda_t=\mathbb{N}$ for all $t\in\overline{I}$,
\eqref{eq:4.18} and the Bochner pseudo-monotonicity of
$\mathbfcal{A}:D(\bfA)\subseteq \mathbfcal{X}\rightarrow
\mathbfcal{X}^*$ with $D(\bfA) =\mathbfcal{X}\cap_{\boldsymbol{j}}\mathbfcal{Y}$, Lemma \ref{3.6} (i) finally provides
$\mathbfcal{A}\boldsymbol{y}=\boldsymbol{\xi}\text{ in
}\mathbfcal{X}^*$. This completes the proof of Theorem
\ref{4.1}.\hfill$\square$

\section{Examples}
\label{sec:5}
In this section we give two prototypical examples to which Theorem \ref{4.1} and the notions developed in Section \ref{sec:3} can be applied. We emphasize that the existence results given in these two examples are not new, see e.g. \cite{Lio69,BR17}. The following examples shall merely illustrate that the
conditions (\hyperlink{C.1}{C.1})--(\hyperlink{C.5}{C.5}) are easily
verifiable and quite general, and in what way the scope of application
is extended by the treatment of pre-evolution triples.

\begin{expl}[Unsteady p-Navier-Stokes equation for $p\ge \frac{11}{5}$]\label{5.1}
  Let $\Omega\subseteq \mathbb{R}^3$ be a bounded domain,
  $I:=\left(0,T\right)$, with $0<T<\infty$. Moreover, let
  $\mathcal{V}=\{\bv\in C_0^\infty(\Omega)^3\mid \text{div }\bv\equiv 0\}$,
  $V$ the closure of $\mathcal{V}$ with respect to
  $\|\nabla\cdot\|_{L^p(\Omega)}$, $H$ the closure of
  $\mathcal{V}$ with respect to $\|\cdot\|_{L^2(\Omega)}$ and let
  $S,B: V\to V^*$ be defined as in the introduction. Then,
  $(V,H,\text{id})$ is an evolution triple, $A:=S+B:V\to V^*$
  satisfies (\hyperlink{C.1}{C.1})--(\hyperlink{C.5}{C.5}) with
  $\mathscr{C}\equiv 0$ in (\hyperlink{C.5}{C.5}) and its induced operator
  $\mathbfcal{A}:L^p(I,V)\cap L^\infty(I,H)\to L^{p'}(I,V^*)$ is
  bounded, Bochner pseudo-monotone and coercive. In addition, for
  arbitrary $\bu_0\in H$ and
  $\boldsymbol{f}\in L^{p'}(I,V^*)$ there exists a solution
  $\boldsymbol{u}\in W_e^{1,p,p'}(I,V,V^*)$ of
  \begin{align*}
      -\int_I{(\boldsymbol{u}(s),\bv)_H\,\varphi\prime(s)\,ds}
      &+\int_I{\int_{\Omega}{\big(\bS(\bD\boldsymbol{u}(s))-\boldsymbol{u}(s)
        \otimes\boldsymbol{u}(s)\big ):\bD
          \bv\,\varphi(s)\,dx}\,ds}
      \\
      &=\int_I{\left\langle\boldsymbol{f}(s),\bv\right\rangle_V\varphi(s)\,ds}
  \end{align*}
  for all $\bv\in V$ and $\varphi\in C_0^\infty(I)$ with
  $\boldsymbol{u}(0)=\bu_0$ in $H$.
\end{expl}

\begin{proof}
	Clearly, $(V,H,\text{id})$ forms an evolution triple and $A:V\to V^*$ is bounded, pseudo-monotone and demi-continuous, see e.g. \cite[Section 6]{BR17}, and thus satisfies
	(\hyperlink{C.1}{C.1}), (\hyperlink{C.2}{C.2}) and (\hyperlink{C.4}{C.4}). (\hyperlink{C.5}{C.5}) with $\mathscr{C}\equiv 0$ immediately follows from $\langle S\bv,\bv\rangle_V=\|\bv\|_V^p$ and $\langle B\bv,\bv\rangle_V=0$ for every $\bv\in V$. For (\hyperlink{C.3}{C.3}) we first note that $\|S\bv\|_{V^*}\leq \|\bv\|_V^{p-1}$ and $\|B\bv\|_{V^*}\leq \|\bv\|^2_{L^{2p'}(\Omega)}$ for every $\bv\in V$. 
	
	If $p\ge 3$, then $p-1\ge2$ and $\|\bv\|_{L^{2p'}(\Omega)}\leq c\|\bv\|_V$ for all $\bv\in V$. Thus, using $a^2\leq (1+a)^{p-1}\leq 2^{p-2}(1+a^{p-1})$ for all $a\ge 0$, we obtain $\|B\bv\|_{V^*}\leq c(1+\|\bv\|_V^{p-1})$ for every $v\in V$.
	
	If $p\in[\frac{11}{5},3)$, then by interpolation with $\frac{1}{\rho}=\frac{1-\theta}{p^*}+\frac{\theta}{2}$, where $\rho=p\frac{5}{3}$, $\theta=\frac{2}{5}$ and $p^*=\frac{3p}{3-p}$, and using $a^{6/5}\leq (1+a)^{p-1}\leq 2^{p-2}(1+a^{p-1})$ for all $a\ge 0$, as $\frac{6}{5}\leq p-1$, we obtain for all $v\in V$ 
	\begin{align*}
	\|\bv\|_{L^{\rho}(\Omega)^3}^2\leq \|\bv\|_H^{\frac{4}{5}}\|\bv\|_{L^{p^*}(\Omega)}^{\frac{6}{5}}\leq \|\bv\|_H^{\frac{4}{5}}(1+\|\bv\|_{L^{p^*}(\Omega)})^{p-1}\leq  c\|\bv\|_H^{\frac{4}{5}}(1+\|\bv\|_V^{p-1}).
	\end{align*}
	Since also $\rho\ge 2p'$, we infer $\|B\bv\|_{V^*}\leq c\|\bv\|_H^{\frac{4}{5 }}(1+\|\bv\|_V^{p-1})$ for every $v\in V$.
	Altogether, $A:V\to V^*$ satisfies (\hyperlink{C.3}{C.3}) and meets the framework of Proposition \ref{3.1} and Corollary \ref{4.2}, which in turn yield the assertion.\hfill$\square$
\end{proof}

\begin{expl}[Unsteady p-Laplace equation with compact perturbation for $p\in (1,\infty)$]\label{5.2}
  Let $\Omega\subseteq \mathbb{R}^d$ be a bounded domain, $I:=\left(0,T\right)$, with $0<T<\infty$,
  and let $b:I\times \Omega\times\mathbb{R}\rightarrow \mathbb{R}$ be a
  function with the following properties:
  \begin{description}[\textbf{(B.1)}]
  \item[\textbf{(B.1)}]\hypertarget{B.1}{} $b$ is measurable in its
    first two variables and continuous in its third variable.
  \item[\textbf{(B.2)}]\hypertarget{B.2}{} For some non-negative
    functions $C_1\in L^{p'}(I,L^q(\Omega))$, $q=\min \{2, (p^*)'\}$, $C_2\in
    L^{p'}(I,L^\infty(\Omega))$ and $1\leq r<
    \max\{2,p\frac{d+2}{d}\}$ holds 
    \begin{align*}
      \vert b(t,x,s)\vert\leq C_1(t,x)+C_2(t,x)(1+\vert s\vert)^{r-1}
    \end{align*}
    for almost every $(t,x)\in I\times\Omega$ and all
    $s\in \mathbb{R}$.
  \item[\textbf{(B.3)}]\hypertarget{B.3}{} For some functions
    $c_1\in L^1(I,L^\infty(\Omega))$ and
    $c_2\in L^1(I\times\Omega,\mathbb{R}_{\ge 0})$ holds
    \begin{align*}
      b(t,x,s)\cdot s\ge c_1(t,x)\vert s\vert^2-c_2(t,x)
    \end{align*}
    for almost every $(t,x)\in I\times\Omega$ and all $s\in \mathbb{R}$.
  \end{description}
	Moreover, set $V=W^{1,p}_0(\Omega)$, $H=L^2(\Omega)$ and let $A_0,B(t):V\cap H\to (V\cap H)^*$ be given via
	\begin{align*}
	\langle A_0v,w\rangle_{V\cap H}
	&:=\int_\Omega{\vert \nabla v\vert^{p-2}\nabla v\cdot \nabla w
		\,dx},\\\langle B(t)v,w\rangle_{V\cap
		H}&:=\int_\Omega{b(t,x,v)\cdot w\,dx} 
	\end{align*}
	for almost every $t\in I$ and all $v,w\in V\cap H$. Then, $(V,H,\text{id})$ forms a pre-evolution triple, $A(t):=A_0+B(t):V\cap H\to (V\cap H)^*$, $t\in I$, satisfies (\hyperlink{C.1}{C.1})--(\hyperlink{C.5}{C.5}) with $\mathscr{C}(s)=s^2$ in (\hyperlink{C.5}{C.5}) and its induced operator $\mathbfcal{A}:L^p(I,V)\cap L^\infty(I,H)\to L^{p'}(I,V^*)$ is bounded, Bochner pseudo-monotone and Bochner coercive. In addition, for arbitrary $\boldsymbol{u}_0\in L^2(\Omega)$ and
  $\boldsymbol{f}\in L^{p'}(I,(V\cap H)^*)$
  there exists a solution
  $\boldsymbol{u}\in W_e^{1,p,p'}(I,V\cap H,(V\cap H)^*)$ of
  \begin{align}
    \begin{split}
      -\int_I{(\boldsymbol{u}(s),v)_H\varphi\prime(s)\,ds}&+\int_I{\int_{\Omega}{
          \big[\vert\nabla\boldsymbol{u}(s)\vert^{p-2}\nabla\boldsymbol{u}(s)\cdot
          \nabla  v+b(t,x,\boldsymbol{u}(s))\cdot
          v\big]\varphi(s)\,dx}\,ds}
      \\
      &=\int_I{\left\langle\boldsymbol{f}(s),v\right\rangle_{V\cap H}\varphi(s)\,ds}
    \end{split}\label{eq:5.3}
  \end{align}
  for all $v\in V\cap H$ and
  $\varphi\in C_0^\infty(I)$ with $\boldsymbol{u}(0)=\boldsymbol{u}_0$
  in $H$. As $C_0^\infty(\Omega)$ is dense in
  $V\cap H$ we infer from \eqref{eq:5.3} the
  distributional identity
  \begin{align*}
    \partial_t\boldsymbol{u}-\textrm{div }(\vert
    \nabla\boldsymbol{u}\vert^{p-2}\nabla\boldsymbol{u})+b(\cdot,\cdot,\boldsymbol{u})=\boldsymbol{f}\quad\text{
    in }\mathcal{D}'(I\times\Omega). 
  \end{align*}
\end{expl}

\begin{proof} Clearly,
  $(V,H,\text{id})=(W^{1,p}_0(\Omega),L^2(\Omega),L^1(\Omega),\text{id},\text{id})$
  forms a pre-evolution triple. As \linebreak
  ${A_0:V\cap H\rightarrow (V\cap H)^*}$ already meets the framework of
  Corollary \ref{4.2} (cf. \cite[Lemmata 1.26 and 1.28]{Ru04}), it
  remains to ensure that ${B(t):V\cap H\rightarrow (V\cap H)^*}$ satisfies (\hyperlink{C.1}{C.1})--(\hyperlink{C.4}{C.4})
  and $A(t):V\cap H\rightarrow (V\cap H)^*$ the semi-coercivity
  condition (\hyperlink{C.5}{C.5}) with $\mathscr{C}(s)=s^2$. We
  restrict ourselves to the case $p \in (1, \frac{2d}{d+2}\big ]$,
  since the case ${p>\frac{2d}{d+2}}$ is already treated in
  \cite[Theorem 6.2]{BR17} or \cite[Proposition 8.37]{Rou05} with the
  help of the evolution triple $(V,H,\text{id})$ and requires only obvious
  modifications to adjust to our framework. From (\hyperlink{B.1}{B.1})
  and (\hyperlink{B.2}{B.2}) in conjunction with the theory of
  Nemyckii operators {(cf.~\cite[Theorem 1.43]{Rou05})} we deduce for
  almost every $t\in I$ the well-definedness and continuity of
  ${F_t:L^\rho(\Omega)\rightarrow L^2(\Omega)}$, where
  ${\rho:=\max\{1,2(r-1)\}}$, given via $(F_tv)(x):=b(t,x,v(x))$ for
  almost every $x\in\Omega$ and all $v\in L^\rho(\Omega)$. In fact, using
  (\hyperlink{B.2}{B.2}) and $(1+ a)^{2(r-1)}\leq (1+ a)^\rho$ for
  all $a\ge 0$, we obtain
  \begin{align}
    \begin{split}
      \|F_tv\|_{L^2(\Omega)}
      &\leq\|C_1(t,\cdot)\|_{L^2(\Omega)}+\|C_2(t,\cdot)\|_{L^\infty(\Omega)}\left(\int_{\Omega}{(1+\vert
          v\vert)^{2(r-1)}\,dx}\right)^{\frac{1}{2}}
      \\
      &\leq
      \|C_1(t,\cdot)\|_{L^2(\Omega)}+\|C_2(t,\cdot)\|_{L^\infty(\Omega)}\big
      (C(\Omega)+\|v\|_{L^{\rho}(\Omega)}^{\frac{\rho}{2}} \big )
    \end{split}\label{eq:5.4}
  \end{align}
  for almost every $t\in I$ and all $v\in V\cap H$.  Due to $\rho<2$,
  $V\hookrightarrow\hookrightarrow L^1(\Omega)$ and Vitali's theorem
  we get $V\cap H\hookrightarrow\hookrightarrow L^\rho(\Omega)$,
  i.e., $\text{id}_{V\cap H}:V\cap H \to L^\rho(\Omega)$ is strongly continuous. From the
  latter, $B(t)=(\text{id}_{V\cap H})^*F_t(\text{id}_{V\cap H})$ and
  the continuity of both
  ${(\text{id}_{V\cap_H})^*:H\rightarrow (V\cap H)^*}$ and
  ${F_t:L^\rho(\Omega)\rightarrow L^2(\Omega)}$ we infer that
  $B(t):V\cap H\rightarrow (V\cap H)^*$ is strongly continuous and
  thus pseudo-monotone. Thus, we verified
  (\hyperlink{C.1}{C.1}), (\hyperlink{C.4}{C.4}) and
  (\hyperlink{C.3}{C.3}) with $\mathscr{B}(s):=s^{\frac{\rho}{2}}$,
  $\alpha(t):=\|C_2(t,\cdot)\|_{L^\infty(\Omega)}$, $\beta(t):= 0$ and
  $\gamma(t):=\|C_1(t,\cdot)\|_{L^2(\Omega)}+C(\Omega)\|C_2(t,\cdot)\|_{L^\infty(\Omega)}$
  (cf.~\eqref{eq:5.4}). Condition (\hyperlink{C.2}{C.2}) is a
  consequence of Fubini's theorem. Using (\hyperlink{B.3}{B.3}), the
  semi-coercivity condition (\hyperlink{C.5}{C.5}) follows by
  \begin{align*}
    \langle A(t)v,v\rangle_V=\langle A_0v,v\rangle_V+\langle
    B(t)v,v\rangle_V\ge
    \|v\|_V^p-\|c_1(t,\cdot)\|_{L^\infty(\Omega)}\|v\|_H^2-\|c_2(t,\cdot)\|_{L^1(\Omega)} 
  \end{align*}
  for almost all $t\in I$ and all $v\in V\cap H$. Altogether,
  $A(t):=A_0+B(t):V\cap H\rightarrow (V\cap H)^*$, $t\in I$, meets the
  framework of Proposition \ref{3.1} and Corollary \ref{4.2}, which yield the
  assertion.\hfill$\square$
\end{proof}

\subsection*{Acknowledgments}
We would like to thank the referee for the helpful comments which
improved the presentation of the paper. 


\begin{thebibliography}{10}

\bibitem{BR17}
{\sc E.~B{\"a}umle and M.~R{\r{u}}{\v{z}}i{\v{c}}ka}, {\sl Note on the
  existence theory for evolution equations with pseudo-monotone operators},
  Ric. Mat. {\bf 66} (2017), no.~1, 35–--50.

\bibitem{BS88}
{\sc C.~Bennett and R.~Sharpley}, {\sl Interpolation of operators}, Pure and
  Applied Mathematics, vol. 129, Academic Press, Inc., Boston, MA, 1988.

\bibitem{BF13}
{\sc F.~Boyer and P.~Fabrie}, {\sl Mathematical tools for the study of the
  incompressible {N}avier-{S}tokes equations and related models}, Applied
  Mathematical Sciences, vol. 183, Springer, New York, 2013.

\bibitem{Bre68}
{\sc H.~Br\'ezis}, {\sl {\'E}quations et in\'equations non lin\'eaires dans les
  espaces vectoriels en dualit\'e}, Annales de l'Institut Fourier {\bf 18}
  (1968), no.~1, 115--175 (fr).

\bibitem{Bro63}
{\sc F.~E. Browder}, {\sl Nonlinear elliptic boundary value problems}, Bull.
  Amer. Math. Soc. {\bf 69} (1963), no.~6, 862--874.

\bibitem{Bro68}
{\sc F.~E. Browder}, {\sl Nonlinear maximal monotone operators in Banach space},
  Mathematische Annalen {\bf 175} (1968), no.~2, 89--113.

\bibitem{GGZ74}
{\sc H.~Gajewski, K.~Gr\"oger, and K.~Zacharias}, {\sl Nichtlineare
  Operator\-glei\-chungen und Operatordifferentialgleichungen},
  Akademie-Verlag, Berlin, 1974.

\bibitem{Hal80}
{\sc J.~Hale}, {\sl Ordinary differential equations / [by] Jack K. Hale},
  SERBIULA (sistema Librum 2.0) (1980).

\bibitem{Hir1}
{\sc N.~Hirano}, {\sl Nonlinear {V}olterra equations with positive kernels},
  Nonlinear and convex analysis ({S}anta {B}arbara, {C}alif., 1985), Lecture
  Notes in Pure and Appl. Math., vol. 107, Dekker, New York, 1987, pp.~83--98.

\bibitem{Hir2}
{\sc N.~Hirano}, {\sl Nonlinear evolution equations with nonmonotonic perturbations},
  Nonlinear Anal. {\bf 13} (1989), no.~6, 599--609.

\bibitem{alex-master}
{\sc A.~Kaltenbach}, {\sl Verallgemeinerte nichtlineare
  {E}volutionsgleichungen}, Master's thesis, Institute of Applied Mathematics,
  Albert-Ludwigs-University Freiburg, 2019.

\bibitem{LM87}
{\sc R.~Landes and V.~Mustonen}, {\sl A strongly nonlinear parabolic
  initial-boundary value problem}, Ark. Mat. {\bf 25} (1987), no.~1, 29--40.

\bibitem{Lio69}
{\sc {J. L.} Lions}, {\sl Quelques M\'ethodes de R\'esolution des Probl\`emes
  aux Limites Non Lin\'eaires}, Dunod, Paris, 1969.

\bibitem{Min63}
{\sc G.~J. Minty}, {\sl On a "Monotonicity" Method for the Solution of
  Nonlinear Equations in Banach Spaces}, Proceedings of the National Academy of
  Sciences of the United States of America {\bf 50} (1963), no.~6, 1038--1041.

\bibitem{Pap97}
{\sc N.~S. Papageorgiou}, {\sl On the Existence of Solutions for Nonlinear
  Parabolic Problems with Nonmonotone Discontinuities}, Journal of Mathematical
  Analysis and Applications {\bf 205} (1997), no.~2, 434 -- 453.

\bibitem{RT01}
{\sc J.~M. Rakotoson and R.~Temam}, {\sl An optimal compactness theorem and
  application to elliptic-parabolic systems}, Appl. Math. Lett. {\bf 14}
  (2001), no.~3, 303--306.

\bibitem{Rou05}
{\sc T.~Roub{\'{\i}}{\v{c}}ek}, {\sl Nonlinear partial differential equations
  with applications}, International Series of Numerical Mathematics, vol. 153,
  Birkh\"auser Verlag, Basel, 2005.

\bibitem{Ru04}
{\sc M.~R{\r u}{\v z}i{\v c}ka}, {\sl Nonlinear functional analysis. An
  introduction. ({N}ichtlineare {F}unktionalanalysis. {E}ine {E}inf\"uhrung.)},
  Berlin: Springer. xii, 208~p., 2004 (German).

\bibitem{Shi97}
{\sc N.~Shioji}, {\sl Existence of periodic solutions for nonlinear evolution
  equations with pseudomonotone operators}, Proc. Amer. Math. Soc. {\bf 125}
  (1997), no.~10, 2921--2929.

\bibitem{Sho97}
{\sc R. E. Showalter}, {\sl Monotone operators in {B}anach space and nonlinear
  partial differential equations}, Mathematical Surveys and Monographs,
  vol.~49, American Mathematical Society, Providence, RI, 1997.

\bibitem{Yos80}
{\sc K.~Yosida}, {\sl Functional analysis}, Springer, Berlin, 1980.

\bibitem{Zei90A}
{\sc E.~Zeidler}, {\sl Nonlinear functional analysis and its applications.
  {II}/{A}}, Springer, New York, 1990, Linear monotone operators.

\bibitem{Zei90B}
{\sc E.~Zeidler}, {\sl Nonlinear functional analysis and its applications. {II}/{B}},
  Springer, New York, 1990, Nonlinear monotone operators.

\end{thebibliography}

\ifx\undefined\bysame
\newcommand{\bysame}{\leavevmode\hbox to3em{\hrulefill}\,}
\fi

\end{document}